\documentclass{amsart}

\usepackage{amsthm} 
\usepackage{color}
\usepackage{amsmath,amscd, amssymb}

\usepackage{enumerate}
\usepackage{graphicx, overpic}
\usepackage{amsmath, amssymb, latexsym, euscript}
\usepackage{url}
\usepackage[all]{xy}

\usepackage[
pdfborderstyle={},
pdfborder={0 0 0},
pagebackref,
pdftex]{hyperref}

\renewcommand*{\backref}[1]{}
\renewcommand*{\backrefalt}[4]{
  \ifcase #1 %
   [No citations.]%
  \or
   [#2]%
  \else
   [#2]%
  \fi
}

 \let\oldmarginpar\marginpar
 \renewcommand\marginpar[1]{\oldmarginpar[\raggedleft\footnotesize #1]%
 {\raggedright\footnotesize #1}}

\def\smallskip{\vspace{.15cm}}
\def\medskip{\vspace{.3cm}}
\def\text{\mbox}
\def\rh2{{\mathbb R}{\mathbb H}^2}
\def\RR{{\mathbb R}}
\def\core{\operatorname{Core}}

\def\ZZ{{\mathbb Z}}
\def\HH{{\mathbb H}}

\def \NN{{\mathbb N}}
\def \CC{{\mathbb C}}

\def\RP2{{\mathbb{RP}}^2}
\def\RP3{{\mathbb{RP}}^3}

\def\PSL{\operatorname{PSL}}

\def\PGL3{PGL(3{,\mathbb R})}
\def\PGL4{PGL(4{,\mathbb R})}

\def\H2R{{\mathbb H}^2\times {\mathbb R}}

\def\Isom{\operatorname{Isom}}
\def\interior{{\rm int}\thinspace}

\def\Dcal{\mathcal D}

\def\Th{\operatorname{Th}}
\def\Hom{\operatorname{Hom}}\def\devGX{\mathcal Dev}

\def\cl{\rm{cl}\thinspace}

\def\PSL{\rm{PSL}}

\def\qed{ $\sqcup\!\!\!\!\sqcap$}

\def\Acal{\mathcal A}

\def\Ccal{{\mathcal C}}
\def\Fcal{{\mathcal F}}

\def\Jcal{\mathcal J}
\def\Lcal{\mathcal L}
\def\Mcal{\mathcal M}
\def\Ncal{\mathcal N}

\def\Qcal{{\mathcal Q}}

\def\Scal{{\mathcal S}}
\def\Tcal{{\mathcal T}}

\def\Vcal{\mathcal V}
\def\Wcal{\mathcal W}

\def\z2{{\mathbb Z}/2}

\newcommand{\area}{\operatorname{area}}
\def\cl{\operatorname{cl}}
\def\CH{\operatorname{CH}}
\def\prefab{Z}

\def\Core{\operatorname{Core}}

\newcommand{\bdy}{{\partial}}
\newcommand{\isom}{{\operatorname{Isom}}}
\renewcommand{\setminus}{\smallsetminus}
\newcommand{\neb}{{\mathcal N}}
\newcommand{\len}{\operatorname{len}}
\newcommand{\Zunion}{{Z}}
\newcommand{\Qthick}{{Q}}
\newcommand{\from}{\colon} 

\newtheorem{theorem}{Theorem}[section]
\newtheorem{lemma}[theorem]{Lemma}
\newtheorem{corollary}[theorem]{Corollary}
\newtheorem{proposition}[theorem]{Proposition}

\newtheorem{claim}[theorem]{Claim}

\theoremstyle{definition}

\newtheorem{remark}[theorem]{Remark}
\newtheorem{definition}[theorem]{Definition}

\newcommand{\refthm}[1]{Theorem~\ref{Thm:#1}}
\newcommand{\reflem}[1]{Lemma~\ref{Lem:#1}}
\newcommand{\refprop}[1]{Proposition~\ref{Prop:#1}}
\newcommand{\refcor}[1]{Corollary~\ref{Cor:#1}}
\newcommand{\refrem}[1]{Remark~\ref{Rem:#1}}
\newcommand{\refclaim}[1]{Claim~\ref{Claim:#1}}

\newcommand{\refitm}[1]{\eqref{Itm:#1}}
\newcommand{\refdef}[1]{Definition~\ref{Def:#1}}
\newcommand{\refsec}[1]{Section~\ref{Sec:#1}}
\newcommand{\reffig}[1]{Figure~\ref{Fig:#1}}

\begin{document}
\title[Ubiquitous quasi-Fuchsian surfaces  in cusped hyperbolic $3$--manifolds]{Ubiquitous quasi-Fuchsian surfaces \\ in cusped hyperbolic $3$--manifolds}
\date{\today}
\author{Daryl Cooper} 
\address{Department of Mathematics, University of California, Santa Barbara, CA 93106}
\email{cooper@math.ucsb.edu}

\author{David Futer}
\address{Department of Mathematics, Temple University, Philadelphia,
	PA 19122}
\email{dfuter@temple.edu}

\begin{abstract}
This paper proves that every finite volume hyperbolic $3$--manifold $M$ contains a ubiquitous collection of closed, immersed, quasi-Fuchsian surfaces. These surfaces are \emph{ubiquitous} in the sense that their preimages in the universal cover separate any pair of disjoint, non-asymptotic geodesic planes. The proof relies in a crucial way on the corresponding theorem of Kahn and Markovic for closed $3$--manifolds. As a corollary of this result and a companion statement about surfaces with cusps, we recover Wise's theorem that the fundamental group of $M$ acts freely and cocompactly on a CAT(0) cube complex.
\end{abstract}

\thanks{ {Cooper was partially supported by NSF grants DMS--1065939, 1207068 and  1045292.}
{Futer was partially supported by NSF grant DMS--1408682.} 
{Both authors acknowledge support from NSF grants DMS--1107452, 1107263, 1107367 ``RNMS: GEometric structures And Representation varieties'' (the GEAR Network). Both authors also acknowledge support from the Institute for Advanced Study.} }

\subjclass[2010]{57M50, 30F40, 20H10, 20F65}

\maketitle

\section{Introduction}\label{Sec:Intro}

A collection 
of immersed surfaces in a hyperbolic $3$--manifold $M=\HH^3/\Gamma$ is called
\emph{ubiquitous} if, for any pair of hyperbolic planes $\Pi, \Pi' \subset \HH^3$
whose distance is $d(\Pi, \Pi') > 0$, 
 there is some surface $S$
  in the collection
  with an embedded preimage $\widetilde{S} \subset \HH^3$ 
 that separates $\Pi$ from $\Pi'$.
The main new result of this paper is the following.

\begin{theorem}\label{Thm:ClosedQFSurfaces} 
Let $M=\HH^3/\Gamma$ be a complete, finite volume hyperbolic $3$--manifold. Then
the set of closed immersed quasi-Fuchsian surfaces in $M$ is ubiquitous.
 \end{theorem}
 
We refer the reader to \refsec{QFBasics} for the definition of a \emph{quasi-Fuchsian surface} (abbreviated QF). 
Informally, a quasi-Fuchsian subgroup of isometries of $\HH^3$ preserves a small deformation of a totally geodesic hyperbolic plane.
  
\refthm{ClosedQFSurfaces} resolves a question posed by Agol \cite[Problem 3.5]{Cooperfest}. 
The case of \refthm{ClosedQFSurfaces} where $M$ is closed is a theorem of Kahn and Markovic  \cite{KahnM}, and forms a crucial ingredient in our proof of the general case.
 Very recently, Kahn and Wright \cite{KahnWright} proved a version of \refthm{ClosedQFSurfaces} with additional control on the quasiconformal constants of the QF surfaces. Their proof extends  the dynamical methods of Kahn and Markovic \cite{KahnM}, including the good pants homology \cite{KahnMarkovic:Ehrenpreis}.

A \emph{slope} on a torus $T$ is an isotopy class of essential simple closed curves on $T$, or equivalently a primitive homology class (up to sign) in $H_1(T)$. 
When $M$ is a \emph{cusped} hyperbolic $3$--manifold --- that is, non-compact with finite volume ---
it follows from the work of Culler and Shalen \cite{CS1} that $M$ contains at least two embedded 
QF surfaces with cusps, and furthermore that the boundaries of these surfaces have distinct slopes on each cusp torus of $M$.  Masters and Zhang \cite{MZ1, MZ2}, as well as 
Baker and Cooper \cite{BCQFS}, found ways to glue together covers of these cusped QF surfaces 
to produce a closed, immersed QF surface. However, it is not clear whether these constructions can 
produce a ubiquitous collection of QF surfaces. 

If an embedded essential surface $S \subset M$ has all components of $\bdy S$ homotopic to the same slope $\alpha \subset M$, we say that $\alpha$ is an  \emph{embedded boundary slope}.
An \emph{immersed boundary slope} in  $M$ is a slope $\alpha$ in a cusp torus of $M$,
for which there is an integer $m>0$ and an essential immersed surface $S$ whose boundary maps to loops each  homotopic 
to $\pm m\cdot\alpha$. Such a surface $S$ is said to have \emph{immersed slope $\alpha$}.
 We prove the following.

 \begin{theorem}\label{Thm:CuspedSurfacesOneSlope} 
Let $M=\HH^3/\Gamma$ be a cusped hyperbolic $3$--manifold, and let $\alpha$ be a 
slope on a cusp of $M$.  Then the set of cusped quasi-Fuchsian surfaces immersed in $M$ 
with immersed slope $\alpha$ is ubiquitous. 
\end{theorem}

Hatcher showed that a compact manifold bounded
by a torus has only finitely many embedded boundary slopes  \cite{HATCHER}.
 Hass, Rubinstein, and Wang \cite{HRW} (refined by Zhang \cite{ZHANG}) 
showed that  there are only finitely many immersed boundary slopes whose surfaces have bounded Euler characteristic.

Baker \cite{BAKER} gave the first example of a hyperbolic manifold with infinitely many
immersed boundary slopes, while Baker and Cooper \cite{BC0}  showed that all slopes of even numerator in the figure-eight knot complement 
 are virtual boundary slopes. 
Oertel \cite{OERTEL} found a manifold with one cusp so that all slopes are immersed boundary slopes.
Maher \cite{MAHER} gave many families, 
including all 2--bridge knots, for which every slope is an immersed boundary slope.
Subsequently, Baker and Cooper \cite[Theorem 9.4]{BC1}  showed that all slopes of one-cusped manifolds
are immersed boundary slopes. Przytycki and Wise \cite[Proposition 4.6]{Przytycki-Wise:mixed-3manifolds} proved the same result for all slopes of multi-cusped hyperbolic manifolds.

The surfaces constructed in those papers are not necessarily quasi-Fuchsian, because they may contain annuli parallel to the boundary. By contrast, \refthm{CuspedSurfacesOneSlope} produces a ubiquitous collection of QF surfaces realizing every slope as an immersed boundary slope.

\subsection{Applications to cubulation and virtual problems}
Theorems~\ref{Thm:ClosedQFSurfaces} and~\ref{Thm:CuspedSurfacesOneSlope} have an application to the study  of $3$--manifold groups acting on CAT(0) cube complexes. We refer the reader to \refsec{Cubulation} for the relevant definitions. 


Theorems~\ref{Thm:ClosedQFSurfaces} and~\ref{Thm:CuspedSurfacesOneSlope},  combined with
 results of Bergeron and Wise \cite{BergeronWise} and Hruska and Wise \cite{Hruska-Wise:RelativeCocompact}, have the following immediate consequence. 

\begin{corollary}\label{Cor:Cubulation}
Let $M$ be a cusped hyperbolic 3--manifold. Choose a pair of distinct slopes $\alpha(V), \beta(V)$ for every cusp  $V \subset M$.
Then
$\Gamma = \pi_1 (M)$ acts freely and cocompactly on a CAT(0) cube complex $\widetilde X$ dual to finitely many immersed quasi-Fuchsian surfaces $S_1, \ldots, S_k$. Every surface $S_i$ is either closed or has immersed slope $\alpha(V)$ or $\beta(V)$ for one cusp $V \subset M$.

Thus $M$ is homotopy equivalent to a compact non-positively curved cube complex $X = \widetilde X / \Gamma$, whose immersed hyperplanes correspond to immersed quasi-Fuchsian surfaces $S_1, \ldots, S_k$.
\end{corollary}

\refcor{Cubulation} is not new. 
The statement that $\pi_1 M$ acts freely and cocompactly on a cube complex $\widetilde X$ is an important theorem due to Wise  \cite{Wise:Hierarchy}. In Wise's work, this result is obtained in the final step of 
his inductive construction of a virtual quasiconvex hierarchy for $\pi_1 M$.  The purpose of this inductive construction is to establish the following stronger statement \cite[Theorem 17.14]{Wise:Hierarchy}.

\begin{theorem}\label{Thm:VirtuallySpecial}
Let $M$ be a cusped hyperbolic $3$--manifold. Then $M$ is homotopy equivalent to a compact non-positively cube complex $X$, which has a finite special cover $\hat X$. The finite-index special subgroup $\pi_1 \hat X \subset \pi_1 M$ embeds into a right-angled Artin group. 
\end{theorem}

\refthm{VirtuallySpecial} has far-reaching consequences. It implies that the manifold $M$ is virtually fibered \cite{agol:fibering-criterion}, and that the fundamental group $\pi_1 M$  is large, subgroup separable, and linear over $\ZZ$.

The same statements also hold for a closed hyperbolic $3$--manifold $M$. In this setting,  \refcor{Cubulation} was proved by Bergeron and Wise \cite[Theorem 1.5]{BergeronWise}, using the closed case of \refthm{ClosedQFSurfaces} due to Kahn and Markovic \cite{KahnM}. Subsequently, Agol \cite{Agol} gave a direct argument to show that the cube complex $X$ is virtually special, establishing \refthm{VirtuallySpecial} 
for closed manifolds. 

After this paper was first distributed, Groves and Manning \cite{GrovesManning:Quasiconvexity} used the work of Agol \cite{Agol}, combined with older results of Minasyan \cite{Minasyan:GFERF} and Haglund--Wise \cite{Haglund-Wise:Coxeter}, to show that the cube complex $X$ constructed in \refcor{Cubulation} is virtually special. That is, \refcor{Cubulation} implies \refthm{VirtuallySpecial}. Consequently, the results of this paper constitute part of a  direct second-generation proof of virtual specialness and virtual fibering for cusped hyperbolic $3$--manifolds, following the same outline as for closed manifolds.
This is considerably easier and more direct than Wise's proof.

In addition to the existence of the cube complex $X$,
 \refcor{Cubulation} asserts that the surfaces used to cubulate $\pi_1 M$ can be chosen to have prescribed immersed slopes. This statement is also not new, as it was proved by Tidmore \cite[Theorem 1.7]{Tidmore:cocompact}. His proof uses the full strength of Wise's work \cite{Wise:Hierarchy}, including \refthm{VirtuallySpecial} and a relative version of the special quotient theorem.
  These strong tools enable Tidmore to settle some open questions about fundamental groups of mixed $3$--manifolds, including biautomaticity and integrality of $L^2$ betti numbers. Our direct proof of \refcor{Cubulation} 
using  \refthm{CuspedSurfacesOneSlope} and \cite{BergeronWise, Hruska-Wise:RelativeCocompact} can be inserted into Tidmore's argument to provide a more straightforward route to these results.

\subsection{Proof outline and organization}
This paper is organized as follows. In \refsec{Background}, we review QF manifolds, convex thickenings, and geometric estimates during Dehn filling. We also extend the prior work of Baker and Cooper \cite{BC1, BCQFS} to prove the \emph{asymmetric combination theorem}, \refthm{CombAsymmetric}, which roughly says that the convex hull of a union of convex pieces stays  very close to one of the pieces.
In \refsec{Auxiliary}, we prove several useful lemmas and characterize ubiquitous collections of surfaces using the notion of a compact \emph{pancake} (see \refdef{Pancake}).
 In \refsec{KMDrill}, we assemble these ingredients to prove  the following weaker version of Theorems~\ref{Thm:ClosedQFSurfaces} and \ref{Thm:CuspedSurfacesOneSlope}.

\begin{theorem}\label{Thm:CuspedSurfacesManySlopes} 
Let $M = \HH^3 / \Gamma$ be a cusped hyperbolic $3$--manifold.  Let $\alpha_1, \ldots, \alpha_n$ be a collection of slopes on cusps of $ M$.
Then there is a ubiquitous  set of cusped QF surfaces 
 immersed in $M$, with the property that for each $\alpha_i$, at least one cusp of each surface is mapped to 
  a multiple $k_i \alpha_i$. 
\end{theorem}

It is worth observing that \refthm{CuspedSurfacesManySlopes} already implies a weak version of \refcor{Cubulation}, namely \emph{co-sparse} cubulation. See \refcor{WeakCubulation} for a precise statement.

Here is the idea of the proof of \refthm{CuspedSurfacesManySlopes}. First, we perform a large Dehn filling on the cusps of $M$ to produce a closed hyperbolic 3--manifold $N$.
Then, by results of Kahn and Markovic \cite{KahnM} and Agol \cite{Agol}, there is a finite cover $\hat N$ of $N$ that contains a closed, embedded, 
almost geodesic QF surface. 
(See \refrem{Separability} for a way to avoid relying on Agol \cite{Agol}.)
A small convex neighborhood of this surface is a 
 compact
QF manifold $Q \subset \hat N$, with strictly convex boundary. The preimages of the filled cusps form
a collection $\Wcal$ of solid tori in $\hat N$. Gluing these onto $Q$ and thickening gives
a compact convex manifold $\Zunion \subset \hat N$ with strictly convex boundary that is
far from the core geodesics $\Delta \subset \Wcal$. Deleting $\Delta$ gives a finite cover
$\hat M$ of $M$, with the property that
the  hyperbolic metric on $\hat M \setminus \Wcal$ is very close
to the hyperbolic metric on $\hat N \setminus \Wcal$. It follows that $\bdy \Zunion$ is also 
locally convex in $\hat M$,
so $Y = \Zunion \setminus \Delta$ is a convex submanifold of $\hat M$. One now surgers
$\partial Y$ inside $Y$ along disks and annuli in $Y$ running out into the cusps of $\hat M$ to produce an embedded, geometrically finite
incompressible surface $F \subset Y$ without accidental parabolics. It follows that $F$ is quasi-Fuchsian,
and the projection of $F$ into $M$ is an immersed QF surface with cusps. This use of the convex envelope $\Zunion$ is similar 
to the method used by Cooper and Long \cite{SSSSS} to show that most Dehn  fillings of a hyperbolic manifold contain a surface subgroup.

To derive  \refthm{ClosedQFSurfaces} from \refthm{CuspedSurfacesManySlopes}, we need to call upon several results and techniques developed by Baker and Cooper \cite{BCQFS}. 
We review these results in  \refsec{BCGluing}. 
Given enough cusped surfaces,  
one can glue together finite covers of copies of convex thickenings of these surfaces, together with some finite covers
of the cusps of $M$, to create a convex manifold $Z$ called a \emph{prefabricated manifold} (see \refdef{Prefab}). This prefabricated manifold
is immersed in $M$ by a local isometry, and each component of $\bdy Z$ is closed and quasi-Fuchsian. 
Projecting $\bdy Z$ down to $M$ yields a closed, immersed QF surface. 
 A mild variation of this technique proves \refthm{CuspedSurfacesOneSlope} in \refsec{UbiquitousSlope}.

Finally, in \refsec{Cubulation}, we explain how to combine Theorems~\ref{Thm:ClosedQFSurfaces} and~\ref{Thm:CuspedSurfacesOneSlope}  with
 results of Bergeron--Wise \cite{BergeronWise} and Hruska--Wise \cite{Hruska-Wise:RelativeCocompact} to show that 
 $M$ is homotopy equivalent to a compact cube complex.

\subsection{Acknowledgements} We thank Ian Agol, Daniel Groves, Chris Hruska, Jason Manning, Matthew Stover, and Daniel Wise for a number of helpful suggestions.  We also thank the referee for helpful comments. We are grateful to the Institute for Advanced Study, including the Ellentuck Fund and the Lunder Fund, for their generosity and hospitality.
%
%
Particular thanks are due to Chef Michel Reymond for nurturing the lunchtime conversation that led to this work.

\section{Background}\label{Sec:Background}

This section lays out the definitions, conventions, and background material that  are used in subsequent arguments. Almost all of the results stated here are widely known and appear elsewhere in the literature. The one result with any novelty is \refthm{CombAsymmetric}, the \emph{asymmetric combination theorem}. This is a mild generalization of the convex combination theorem of Baker and Cooper \cite{BC1, BCQFS}. The generalized statement described here may be of some independent interest.
 
 \subsection{Convex and complete manifolds}\label{Sec:ConvexBackground}
A \emph{hyperbolic $n$--manifold} is a  smooth $n$--manifold,
possibly with boundary, equipped with a metric so that every point has a neighborhood
that is isometric to a subset of hyperbolic space, $\HH^n$.
A connected hyperbolic $n$--manifold $M$ is \emph{convex} if every pair of points in the universal cover $\widetilde{M}$ is connected
by a geodesic. It is \emph{complete} if the universal cover is isometric to
$\HH^n$. 

We emphasize that the hyperbolic manifolds considered in this paper are not necessarily complete. On the other hand, all manifolds in this paper are presumed   connected and orientable, unless noted otherwise. As we describe at the start of \refsec{QFDrill}, disconnected manifolds are typically denoted with calligraphic letters.

The following facts are straightforward; see \cite[Propositions 2.1 and 2.3]{BC1}.

\begin{lemma}\label{Lem:ConvexProperties}
Let $M$ be a convex hyperbolic $n$--manifold. Then
  the developing map embeds   $\widetilde{M}$ isometrically into $\HH^n$, and the covering transformations
of $\widetilde{M}$ extend to give a group $\Gamma  \subset \isom \,  \HH^n$.
Consequently, $M$ is isometric to a submanifold of $N = \HH^n/\Gamma$, where $N$ is unique up to isometry.  

If $M$ is convex and 
$f \from M \to N$
is a local isometry into a hyperbolic $n$--manifold $N$, then $f$ is $\pi_1$--injective. \qed
\end{lemma}

%
%
%

The geodesic compactification of $\HH^n$ is the closed ball $\overline \HH^n=\HH^n\sqcup \bdy \HH^n$.
In the main case of interest, $n = 3$, we  write $ \bdy \HH^3=S^2_{\infty}$.
The \emph{limit set} of a subset $A\subset \HH^n$ is $\Lambda(A)=\overline{A} \cap  \bdy \HH^n$. If $\Gamma \subset \isom ( \HH^n)$ is a discrete group, then $\Lambda(\Gamma)$ is the limit set of an orbit $\Gamma x$, for an arbitrary $x \in \HH^n$.

The convex hull of a set $A \subset \overline\HH^n$, denoted $\CH(A)$, is the intersection of $\HH^n$ and all the convex subsets containing $A$. 
If $M$ is a convex hyperbolic manifold, then by \reflem{ConvexProperties}, $M$ isometrically embeds into a complete manifold $N = \HH^n / \Gamma$. We define the  \emph{convex core} of $M$ to be $\core(M) =\CH(\Lambda(\Gamma)) / \Gamma$.  
Then $\core(M) = \core(N)$, and  $\core(M)\subset \overline{M} \subset N$.
                 
A convex hyperbolic $n$--manifold is \emph{geometrically finite} if some (any) 
$\epsilon$--neighborhood of $\core(M)$ has finite volume. 
We  focus our attention on two special kinds of geometrically finite hyperbolic $3$--manifolds.

\subsection{Quasi-Fuchsian basics}\label{Sec:QFBasics}
A \emph{Fuchsian group} is a discrete, torsion-free, orientation-preserving subgroup $\Gamma \subset \isom(\HH^2)$, such that the quotient $S = \HH^2 / \Gamma$ has finite area. We call $S$  a \emph{finite area hyperbolic surface}. A Fuchsian group $\Gamma$ stabilizing a hyperbolic plane $\HH^2 \subset \HH^3$ can also be considered a subgroup of $\isom(\HH^3)$. 
Note that the convex core of a Fuchsian group is $\core(\HH^3 / \Gamma) =  \HH^2/\Gamma = S$, which has 3--dimensional volume $0$.

A convex hyperbolic $3$--manifold  $M$ is called \emph{quasi-Fuchsian} (or \emph{QF} for short) if there is a finite-area hyperbolic surface $S$ such that
 $\core(M)$ has finite volume and is homeomorphic $S \times I$ or to $S$. 
    To overcome this mild
    technical irritation, we define a convex $3$--manifold by $Q(S)=\core(M)$ unless $S$ is
    Fuchsian, in which case $Q(S)$ is a small convex neighborhood of $S$.   Thus $Q(S) \cong S \times I$ in all cases.
 
By \reflem{ConvexProperties}, every quasi-Fuchsian $3$--manifold $Q$ isometrically embeds into a complete $3$--manifold $N \cong S \times \RR$. We call $N$ the \emph{universal thickening} of $Q$ (compare \refsec{Thickening}).
  We call
 $S \times \{ 0 \} \subset N$ a \emph{quasi-Fuchsian surface}.
  If $f\from  N \to M$ is 
a covering map of complete hyperbolic $3$--manifolds, and $N$ is quasi-Fuchsian,
the restriction $f \from  S \times \{ 0 \} \to M$ gives an \emph{immersed quasi-Fuchsian surface} in $M$. Note that immersed QF surfaces are automatically $\pi_1$--injective.

Let $S$ be a finite-area hyperbolic surface. 
A representation $\rho\from  \pi_1 S \to \isom(\HH^3)$ is called \emph{type-preserving} if $\rho$ sends peripheral loops in $S$ to parabolic isometries and non-peripheral loops to non-parabolic isometries.
If a loop $\gamma \in \pi_1 S$ is \emph{not} homotopic into a puncture but $\rho(\gamma)$ is parabolic, we say that $\gamma$ is an \emph{accidental parabolic} for $\rho$. Thus type-preserving representations cannot have accidental parabolics.

%

 We will deduce that certain  surfaces are QF via the following classical result.

\begin{theorem}\label{Thm:BonahonCanary}
Let $Y$ be a convex, finite volume hyperbolic $3$--manifold with $\bdy Y \neq \emptyset$. Let $f \from S \to Y$ be an immersion from a hyperbolic surface $S$, such that  $f_* \from \pi_1 S \to \pi_1 Y$ is faithful and type-preserving.  
Then $S$ is quasi-Fuchsian.
\end{theorem}

\begin{proof}
By \reflem{ConvexProperties}, $Y$ has an isometric embedding into a complete $3$--manifold $M$. Let $M_S$ be the cover of $M$ corresponding to $f_* (\pi_1 S) \subset \pi_1 M$. We claim that $M_S$ is geometrically finite. 

Suppose $M_S$ is geometrically infinite. Then Thurston's covering theorem \cite{Canary}, combined with tameness for surface groups \cite{Bonahon:Bouts}, says that $M$ must have finite volume and $f_*(\pi_1 S)$ must be a virtual fiber subgroup of $\pi_1 M$. In this case $\Lambda(\pi_1 M) = \Lambda(f_* (\pi_1 S)) = S^2_\infty$, hence $\core(M) = M$. But $\core(M) \subset Y$ by the convexity of $Y$, hence $\bdy Y = \emptyset$, a contradiction.

Since $M_S$ is geometrically finite and has no accidental parabolics, it is a standard fact that it is quasi-Fuchsian. Thus $S$ is an immersed QF surface in $Y$.
\end{proof}

\subsection{Convex thickenings}\label{Sec:Thickening}
Let $M$ and $N$ be (possibly disconnected) hyperbolic $3$--manifolds with $M \subset N$. 
We say
$N$ is a \emph{thickening} of $M$ if the inclusion $\iota\from  M \hookrightarrow N$ is a homotopy equivalence.
If, in addition, each component of $N$ is convex, then 
 $N$ is called  a \emph{convex thickening} of $M$.  If $M$ is connected and has a convex thickening,
 then the  \emph{convex hull} of $M$ is 
   $\CH(M)=\CH(\widetilde M) / \Gamma$ where $\Gamma$ is the holonomy of $M$.

     Convex thickenings only change the boundary of a $3$--manifold by isotopy.

\begin{lemma}\label{Lem:ConvexCollar} 
Let $Y$ be a (connected) hyperbolic $3$--manifold. Let $Z$ be a  convex thickening of $Y$ with finite volume and incompressible boundary.
Suppose that  every 
end of $Z$ is a rank--$1$ or rank--$2$ cusp that contains a corresponding cusp of $Y$.
Then $\bdy Y$ is isotopic in $Z$ to $\bdy Z$.
\end{lemma}

\begin{proof} 
We construct a compact manifold $Z^- \subset Z$ by cutting off the non-compact ends of $Z$ along a union of horospherical tori $\Tcal$ and a union of horospherical annuli $\Acal$. Set $Y^- \! =Z^-\cap Y$. 
Let $N$ be the 3--manifold obtained by taking two copies of $Z^-$ and gluing them by the identity map along $\Acal$. Let $M\subset N$ be the corresponding
double of $Y^-$ along $\Acal$. Let $\Sigma \subset \bdy N$ be a component
that is not in the double of $\Tcal$. Then $\Sigma$ is the double of part of the boundary of $Z^-$.

 Let $P$ be the closure of the component of $N \setminus M$ that contains $\Sigma$. 
 Since $\bdy Z$ is incompressible,
  $\Sigma$ is a closed surface that is incompressible
 in both $N$ and $P$. Since $Z$ is convex, it is irreducible, hence so  is $P$. The inclusions $Y\hookrightarrow Z$ and $M \hookrightarrow N$ are homotopy equivalences by the definition of ``thickening.''
 By Waldhausen's theorem
      \cite{waldhausen:suff-large},
  $P\cong \Sigma\times I$. The components 
 of $\Acal\cap P$ can be isotoped to be vertical in this product. Thus $\Sigma \cap N$ is properly isotopic into $\bdy Y$.
 \end{proof}

  If $M$ is a subset of a metric space $N$,  the (closed) \emph{$r$--neighborhood} of $M$ in 
$N$ is  
\[
\neb_{r}(M;N)=\{x\in N \, : \, d(x,M)\leq r \} .
\]
We will omit the second argument of $\neb_r(\cdot \, , \cdot)$ when it is clear from context, for instance when the ambient set is $\HH^n$. If $M$ his a convex hyperbolic $n$--manifold, recall from \reflem{ConvexProperties} that the developing map gives an isometric embedding $\widetilde M \to \HH^n$, and identifies $\pi_1 M$ with a discrete subgroup $\Gamma \subset \isom (\HH^n)$. 
We define the \emph{$r$--thickening} of $M$ to be
 \[
 \Th_{r}(M)\: =\:  \neb_{r} \big( \widetilde{M}; \, \HH^n \big) / \Gamma,
\qquad  \text{and set } \qquad
\Th_{\infty}(M)  \: = \: \bigcup_{r > 0} \Th_{r}(M)  \: =\:   \HH^n  / \Gamma.  \]
By construction, the \emph{universal thickening} $\Th_{\infty}(M)$ is complete.

If $S \subset M$ is a smoothly embedded surface and $v \in T_x S$,
 the \emph{extrinsic curvature} of $S$ along $v$ is (the absolute value of) the curvature in $M$ of the geodesic in $S$ that has tangent vector $v$.
Thus totally geodesic surfaces have constant extrinsic curvature $0$, while horospheres have constant extrinsic curvature $1$. 

  Given $\kappa \geq 0$, a surface $S$ in a hyperbolic $3$--manifold $M$ is \emph{$\kappa$--convex} if for each
$x\in S$ there is a smooth surface $D \subset M$ of constant extrinsic curvature $\kappa$, such
  such that $x \in D\cap S$ and $S$ is locally on the convex side of $D$. 
  This is sometimes called  \emph{$\kappa$--convex in the  barrier sense}.  
  When $\kappa=0$, then $D$ is totally geodesic, and we say $S$ is \emph{locally convex} at $x$.
  Thus $\kappa$--convexity implies local convexity. 
Observe that if $\Pi \subset \HH^3$ is a totally geodesic copy of $\HH^2$, then $\bdy \neb_r(\Pi)$
  has constant extrinsic curvature $\kappa= \kappa(r) > 0$. 
This implies

\begin{lemma}\label{Lem:KappaConvex} There is a continuous, monotonically increasing function $\kappa\from  \RR_+ \to(0,1)$ such that
 if $X \subset \HH^3$
is a closed convex set and $r > 0 $, then $\bdy \Th_r(X)$ is 
$\kappa(r)$--convex. \qed
\end{lemma}

 The point of using barrier surfaces in the definition of $\kappa$--convexity is that in most cases, $\bdy \core(M)$ and $\bdy \Th_r (\core(M))$ are not smooth. However, \reflem{KappaConvex} still applies to these surfaces.

The next result is a variant of the convex combination theorem of Baker and Cooper \cite{BC1, BCQFS}. 

\begin{theorem}[Asymmetric combination theorem]
\label{Thm:CombAsymmetric}
For every $\epsilon>0$, there is $R=R(\epsilon)>8$ such that the following holds.
Suppose that
\begin{enumerate}
\item  $Y=Y_0 \cup \ldots \cup Y_m$  and $M = M_0 \cup \ldots \cup M_m$ are connected hyperbolic $3$--manifolds.
\item  $M_i$ and $Y_i$ are convex hyperbolic $3$--manifolds, such that $Y_i$ is a thickening of $M_i$.
\item  $ Y_i \supset  \Th_8(M_i)$.
\item  If $Y_i \cap Y_j \neq \emptyset$, then $i = j$ or one of $i,j$ is $0$.
\item  Every component of $Y_i \cap Y_j$ contains a point of $M_i \cap M_j.$
\end{enumerate}
Then
\(
\CH(M) \: \subset \: \neb_{\epsilon}(M_0; \, Y) \cup \neb_R(M_1 \cup \ldots \cup M_m; \, Y) .
\)
   \end{theorem}

\refthm{CombAsymmetric} strengthens  \cite[Theorem 1.3]{BCQFS} in two small but useful ways. First, it allows a larger number of convex pieces to combine. Second, it records the conclusion that, far away from the pieces $M_1, \ldots, M_m$, the convex hull of $M$  is $\epsilon$--close to $M_0$.

\begin{proof}
First, we verify that 
 \[
 \CH(M) \: \subset \:  \Th_8(M_0)\cup \ldots \cup \Th_8(M_m) \: \subset \: Y.
 \]
When $m = 1$, this is a special case of  \cite[Theorem 1.3]{BCQFS}. (In that theorem, the pair $M_0, M_1$ are denoted $M_1, M_2$.) 
Indeed, hypotheses (C1), (C2) of that theorem are restated in (1), (2). Hypotheses (C3), (C4), (C5) of that theorem are all implied by (2) and (3). Hypothesis (C6) of that theorem is restated in (5). Thus the desired conclusion holds when $m = 1$. 

The proof of \cite[Theorem 1.3]{BCQFS} combines a lemma about thin triangles with  \cite[Theorem 2.9]{BC1} to conclude that $r=8$ is a sufficient thickening constant. Thus almost all of the work is performed in  \cite[Theorem 2.9]{BC1}.

For $m > 1$, we observe that the proof of \cite[Theorem 2.9]{BC1} goes through verbatim when $M_1$ is replaced by the disjoint union $M_1 \sqcup \ldots \sqcup M_m$. Note that condition (4) guarantees that $Y_1, \ldots, Y_m$ are pairwise disjoint, hence
 $M_1, \ldots, M_m$ are as well. This concludes the proof that $\CH(M) \subset Y$.

To prove the stronger containment claimed in the theorem, we need the following easy lemma. We use the notation $[a,b]$ to denote a geodesic segment with endpoints at $a$ and $b$.

\begin{lemma}\label{Lem:FellowTravel}
For every $\epsilon>0$, there is $\hat{R}=\hat{R}(\epsilon) >0$ such that the following properties hold. 
\begin{enumerate}
\item\label{Itm:Tri} For every triangle in $\HH^3$ with vertices $a,b,c$, the geodesic $[a,c]$ satisfies
\[
[a,c] \subset \neb_{\hat{R}}([a,b]) \cup \neb_{\epsilon/4}( [b,c]).
\]
\item\label{Itm:Quad} Every (skew) quadrilateral in $\HH^3$ with vertices $a_1, b_1, b_2, a_2$ satisfies 
\[
[a_1, a_2] \: \subset \:  \neb_{\hat{R}} ([a_1, b_1]) \: \cup \:  \neb_{\hat{R}} ([a_2, b_2]) \: \cup \: \neb_{\epsilon/2} ([b_1, b_2] ).
\]
\end{enumerate}

\end{lemma}

\begin{proof}
We introduce the following terminology. A \emph{rounded rectangle} is a region $P \subset \HH^2$ with four corners at $w,x,y,z$, such that $[x,y]$ is a geodesic segment, $[w,x]$ and $[y,z]$ are geodesic segments of length $h$ perpendicular to $[x,y]$, and the fourth side is an arc from $w$ to $z$ that stays at constant distance $h>0$ from $[x,y]$. The side $[x,y]$ is called the \emph{base} of $P$, and $h$ is called the \emph{height} of $P$. A standard calculation shows that $\area(P) = \sinh(h) \len([x,y])$.

Conclusion \refitm{Tri} is proved by an area argument.
 Let $\Delta$  be the  triangle with vertices $a,b,c$.
If $[a,c] \subset \neb_{\epsilon/4}( [b,c]) $, the conclusion holds trivially. Otherwise, let $P \subset \Delta$ be the maximal rounded rectangle with height $\epsilon/4$ and base $[x,y]$ contained in $[b,c]$. Then, by maximality, at least one corner of $P$ (either $w$ or $x$) lies in $[a,b]$, while another corner
(labeled $z$) is the unique point of $[a,c]$ at distance $\epsilon/4$ from $[b,c]$. Since
\[
 \sinh(\epsilon/4) \len([x,y]) = \area(P)  < \area(\Delta) < \pi,
\]
we conclude that
\[
d([a,b], \, z) \: < \:  \len([w,x]) + \len([x,y]) + \len([y,z]) \: < \: \frac{\epsilon}{4} + \frac{\pi}{\sinh(\epsilon/4)} + \frac{\epsilon}{4}.
\]
Setting $\hat{R}(\epsilon) = \frac{\pi}{\sinh(\epsilon/4)} + \frac{\epsilon}{2}$ yields
\[
[a,z] \subset \neb_{\hat R(\epsilon)}([a,b])
\qquad \text{and} \qquad
[z,c] \subset \neb_{\epsilon/4}([b,c]),
\]
implying \refitm{Tri}.
Now, conclusion \refitm{Quad} follows from \refitm{Tri} by triangle inequalities.
\end{proof}

We now return to the proof of \refthm{CombAsymmetric}.
For any $\epsilon > 0$,  let $R = R(\epsilon) = \hat{R}(\epsilon) + 16$, where $\hat{R}(\epsilon)$ is the function of \reflem{FellowTravel}.
For every $r > 0$, let $Z_r = \neb_r(M_1 \cup \ldots \cup M_m; Y)$. We have already seen that
 \[
 \CH(M) \: \subset \:  \Th_8(M_0)\cup \ldots \cup \Th_8(M_m) 
 \: = \:   \Th_8(M_0) \cup Z_8.
 \]
Since $R > 8$, this implies
\[
\CH(M) \setminus Z_R \:  \subset \:  \CH(M) \setminus Z_8 
\: \subset \:  \CH \big( M_0 \cup (Z_8 \cap \Th_8(M_0)) \big),
\]
where $Z_8 \cap \Th_8(M_0) = \Th_8(M_0 \cap M_1) \cup \ldots \cup \Th_8(M_0 \cap M_m)$ is the thickened intersection.
 To see the second containment, suppose $x\in \CH(M) \setminus Z_8.$ By  Carath\'eodory's theorem  \cite[Proposition 5.2.3]{Papadopoulos},  there is a geodesic $3$--simplex $\Delta$ with vertices at $a,b,c,d\in M$ such that $x \in \Delta \subset\CH(M)$.
Let $[x,a']$ be the maximal sub-interval of $[x,a]$ that is disjoint from the interior of $Z_8$. 
If $a = a'$, then $a' \in M \setminus \interior(Z_8) \subset M_0$. Otherwise, $a \neq a'$, hence $a' \in \bdy Z_8 \cap \Th_8(M_0)$. Thus, in both cases,
\[ 
a'\in  M_0 \cup (Z_8 \cap \Th_8(M_0)).
\]
Define $b',c',d'$ in a similar fashion to $a'$, and let $\Delta'\subset\Delta$ be the geodesic simplex with vertices
$a',b',c',d'$. Then $x\in\Delta'\subset\CH(M_0\cup( Z_8\cap \Th_8(M_0)))$.

Given $x \in   \CH(M) \setminus Z_R $, our goal is to show that $x \in \neb_\epsilon(M_0)$.
The above characterization of $\CH(M) \setminus Z_R$, combined with Carath\'eodory, implies that $x$ lies in a geodesic $3$--simplex  $\Delta$ whose vertices $a_0, \ldots, a_3$ 
are in $M_0 \cup (Z_8 \cap \Th_8(M_0))$. Let $b_i$ be the point of $M_0$ closest to $a_i$. Then $d(a_i, b_i) \leq 8$.
In fact, if $a_i \neq b_i$, then  $a_i \notin M_0$, hence $a_i \in Z_8$, which means that $[a_i, b_i] \subset Z_{16}$. 

Consider how far the segment $[b_i, b_j]$ can be from $[a_i, a_j]$. If $a_i \neq b_i$ but $a_j = b_j = c$, \reflem{FellowTravel}\refitm{Tri} gives
\[
[a_i, a_j]  \: \subset \: \neb_{\hat{R}}([a_i,b_i]) \cup \neb_{\epsilon/4}([b_i,b_j])  
\: \subset \: \neb_{\hat{R}} (Z_{16})  \cup \neb_{\epsilon/2} ([b_i, b_j] ) \: = \: Z_R \cup  \neb_{\epsilon/2} ([b_i, b_j] ).
\]
Similarly, if $a_i \neq b_i$ and $a_j \neq b_j$, 
\reflem{FellowTravel}\refitm{Quad} gives
\begin{align*}
[a_i, a_j] 
& \subset \neb_{\hat{R}} ([a_i, b_i]) \cup \neb_{\hat{R}} ([a_j, b_j]) \cup \neb_{\epsilon/2} ([b_i, b_j] ) \\
& \subset \neb_{\hat{R}} (Z_{16}) \quad \cup \neb_{\hat{R}} (Z_{16}) \quad \cup \neb_{\epsilon/2} ([b_i, b_j] ) \\
& = Z_R \qquad \qquad \cup  \qquad \qquad \neb_{\epsilon/2} ([b_i, b_j] ). 
\end{align*}
Let $\Delta'$ be the simplex in $M_0$ with vertices $b_0, \ldots, b_3$. Then the corresponding sides of $\Delta$ and $\Delta'$ either lie in $Z_R$ or are $\epsilon/2$ fellow-travelers. Since $x \in \Delta \setminus Z_R$, it follows that $d(x,\Delta')<\epsilon$. 
But the convex manifold $M_0$ contains $\Delta'$, hence $x \in \neb_\epsilon(M_0)$
 as desired.
   \end{proof}

\subsection{Cusps, tubes, and Dehn filling}\label{Sec:FillingBasics}
A \emph{cusped} hyperbolic $3$--manifold is one that is complete, non-compact, and with finite volume. Every cusped manifold $M$ can be decomposed into a compact submanifold $A$ and a disjoint union of horocusps. Here, 
a $(3$--dimensional, rank--$2$) \emph{horocusp} is $C=B/\Gamma$, where
 $B\subset \HH^3$ is a horoball and $\Gamma \cong \ZZ \times \ZZ $  is a discrete group of parabolic isometries that
preserve $B$. The boundary $\bdy  C=\bdy  B /\Gamma$ is called a \emph{horotorus}. 

A \emph{tube} is a compact, convex hyperbolic solid torus. Any tube $W$ contains exactly one closed geodesic, called the \emph{core curve} and denoted $\delta(W)$. A \emph{round tube} is a uniform $r$--neighborhood about its core curve.

Given a hyperbolic manifold $M$ and $\epsilon > 0$, the \emph{$\epsilon$--thick part} $M^{\geq \epsilon}$ is the set of all points whose injectivity radius is at least $\epsilon/2$. The \emph{$\epsilon$--thin part} is $M^{\leq \epsilon} = \overline{M \setminus M^{\geq \epsilon}}$.  A famous consequence of the Margulis lemma is that there is a uniform constant $\mu_3$ such that for every $\epsilon \leq \mu_3$ and every complete hyperbolic $3$--manifold $M$, the thin part $M^{\leq \epsilon}$ is a disjoint union of horocusps and round tubes.

Let $M$ be a cusped hyperbolic $3$--manifold.
Given a horocusp $C \subset M$, and a slope $\alpha$ on $\bdy C$, the \emph{length of $\alpha$}, denoted $\ell(\alpha)$, is the length of a Euclidean geodesic representative. The \emph{normalized length} 
of $\alpha$ is the quantity $L(\alpha) = \ell(\alpha)/\sqrt{\area(\bdy C)}$, which is left unchanged when $C$ is expanded or contracted. The definitions of $\ell(\alpha)$ and $L(\alpha)$ extend linearly to non-primitive homology classes in $H_1(\bdy C)$.

For a slope $\alpha$ on $\bdy C$, \emph{Dehn filling $M$ along $\alpha$} is the process of removing a horocusp $C$ and gluing in a solid torus $W$ so that the meridian disk is mapped to $\alpha$. The resulting $3$--manifold is denoted $M(\alpha)$. For an integer $k > 1$, \emph{Dehn filling $M$ along $k\alpha$} produces a $3$--orbifold with base space  $M(\alpha)$ and singular locus of order $k$ along the core curve of the added solid torus $W$. The same definition applies to Dehn fillings of $M$ along multiple horocusps.

Thurston showed that the change in geometry under Dehn filling is controlled by the length of a filling slope \cite{WPT}. 
Hodgson and Kerckhoff made this control much more quantitative \cite{hk:univ, hk:shape}. The following theorem, building on their work, is essentially due to Brock and Bromberg \cite{brock-bromberg:density}.

\begin{theorem}\label{Thm:BilipDrillFill}
Let $\epsilon > 0$, $\kappa > 0$, and $J > 1$ be constants, where $\epsilon \leq \mu_3$. Then there exists a number $K = K(\epsilon, \kappa, J)$ such that the following holds for every cusped hyperbolic $3$--manifold $M$.

Let $C_1, \ldots C_n$ be a disjoint collection of horocusps, where each $C_i$ is a component of $M^{\leq \epsilon}$. Let $A = M \setminus \bigcup C_i$. 
Let $\alpha_1, \ldots, \alpha_n$ be (possibly non-primitive) homology classes on $\bdy C_1, \ldots, \bdy C_n$. Then, for any Dehn filling in which each $\alpha_i$ satisfies $L(\alpha_i) \geq K \sqrt{n}$, we have the following.
\begin{enumerate}
\item $N = M(\alpha_1, \ldots, \alpha_n)$ has a complete hyperbolic metric, in which the cores $\delta_1, \ldots, \delta_n$ of the added solid tori are closed geodesics.
\item\label{Itm:Bilip} There is a diffeomorphism $\varphi\from  M \to N \setminus (\delta_1\cup  \ldots \cup \delta_n)$, whose restriction to $A$ is $J$--bilipschitz.
\item\label{Itm:CurvBound} Let $S \subset A$ be a surface. If $S$ is $\kappa$--convex, then the image $\varphi(S)$ is $\kappa/2$--convex. Conversely, if $\varphi(S)$ is $\kappa$--convex, then $S$ is $\kappa/2$--convex.
\end{enumerate}
\end{theorem}

\begin{proof}
Brock and Bromberg proved the same result under hypotheses on $N$ rather than $M$ (one needs to assume that the total length of the cores  $\delta_1, \ldots, \delta_n \subset N$ is sufficiently small). See \cite[Theorem 1.3]{brock-bromberg:density} for the $J$--bilipschitz diffeomorphism and \cite[Corollary 6.10]{brock-bromberg:density} for the control of extrinsic curvature. (See also \cite{brock-bromberg:erratum}.) Using estimates by Hodgson and Kerckhoff \cite{hk:univ}, Magid translated their  result into hypotheses on normalized length in $M$ \cite[Theorem 1.2]{magid:deformation}. As Hodgson and Kerckhoff clarify in \cite{hk:shape}, the methods apply equally well to a simultaneous orbifold filling of multiple cusps, so long as each $\alpha_i$ has normalized length $L(\alpha_i) \geq K \sqrt{n}$. See the Remark containing \cite[Equation (37)]{hk:shape}.
\end{proof}

We remark that the length cutoff $K = K(\epsilon, \kappa, J)$ in \refthm{BilipDrillFill} is independent of the manifold $M$. The dependence of $K$ on the constants $\epsilon, \kappa, J$ is made explicit in forthcoming work of Futer, Purcell, and Schleimer \cite{FPS:EffectiveBilipschitz}. We will not need this here. In fact, \refthm{BilipDrillFill} is already stronger than what we need; see \refrem{NoBB}.

\subsection{Covers, lifts, and elevations}\label{Sec:Elevation}

Throughout the paper, we  deal with immersed objects in hyperbolic $3$--manifolds, as well as their preimages in 
 (finite or infinite) covering spaces. This requires some careful terminology.

Suppose that $X$ and $M$ are manifolds,  $p\from  \hat M \to M$ is a  covering map,  and $f\from  X \to M$ is an immersion. If $f_* (\pi_1 X) \subset p_* (\pi_1 \hat M)$, the map $f$ lifts to $\hat{f}\from  X \to \hat M$. Such a lift is determined by a local inverse to $p$ at a point of $f(X)$.
We call $\hat{f}$ the \emph{lift of $f$} and the image $\hat f(X)$ the \emph{lift of $X$} in $\hat M$.

More generally, if $\pi\from  \widetilde{X} \to X$ is the universal covering map, then $f \circ \pi$ always has a lift $\widetilde{f}\from  \widetilde{X} \to \hat M$. Again, the lift is determined by a local inverse to $p$ at a point of $f(X)$. We call the image $\widetilde{f}(\widetilde{X})$ an \emph{elevation} of $X$ in $\hat{M}$. If   $p\from  \hat M \to M$ is a finite cover, an elevation of $X$ in $\hat M$ is a lift of some finite cover of $X$. However, if $f$ is not $1$--$1$, an elevation of $X$ may fail to be a component of $p^{-1}(X)$.

A subgroup $H \subset G$ is called \emph{separable} if $H$ is the intersection of finite index subgroups of $G$. The group $G$ is called \emph{residually finite} if $\{ 1\}$ is separable, and \emph{subgroup separable} 
if all finitely generated subgroups are separable.
A deep observation of Scott \cite{Scott:LERF} 
is that  if $X$ is compact and $f\from X \to M$ is an immersion that lifts to an embedding in some infinite cover of $M$,
and  $f_* ( \pi_1X)$ is separable in $\pi_1(M)$, then $f$ also lifts to an embedding 
 $\hat{f}\from  X \to \hat{M}$ into a finite cover $\hat M$.

Hyperbolic manifold groups are residually finite by Selberg's lemma.
Scott showed that the fundamental groups of surfaces are subgroup separable \cite{Scott:LERF}. Agol  showed that all (finitely generated) fundamental groups of hyperbolic $3$--manifolds are subgroup separable \cite[Corollary 9.4]{Agol}, completing a program developed by Wise \cite{Wise:Hierarchy, Wise:Riches2Raags}. Our argument in \refsec{KMDrill} uses Agol's theorem, although this is mainly a matter of convenience; see \refrem{Separability}. 
The argument  in \refsec{UbiquitousSlope} uses subgroup separability in surfaces, 
and draws on the previous work of Baker and Cooper that does the same.

\section{Fat tubes, thin surfaces, and pancakes}\label{Sec:Auxiliary}

This section lays out some elementary results that  are needed in the proof of \refthm{CuspedSurfacesManySlopes}. In \refsec{NiceProduct}, we explore the notion of quasi-Fuchsian manifolds with \emph{nice product structures}, and show that  collared geodesics in such manifolds are naturally classified into three types (\refdef{SkirtMerLong}). In \refsec{Pancake}, we characterize ubiquitous collections of surfaces in a $3$--manifold using the notion of a \emph{pancake} (\refdef{Pancake}). The advantage of this point of view is that pancakes are compact objects, hence are well-behaved under Dehn filling.

\subsection{Nice product structures}\label{Sec:NiceProduct}

A \emph{product structure} on a quasi-Fuchsian manifold $Q \cong S \times I$ is a diffeomorphism
 $f\from S\times[-1,1]\rightarrow Q$. The two boundary components of $Q$ are $\bdy_{\pm}Q=f(S\times \{ \pm 1 \})$.
The arc $f(x\times[-1,1])$ is called \emph{vertical}.
A product structure determines a map $\pi_h\from Q\rightarrow S$ called \emph{horizontal projection}.
 The \emph{midsurface} of $Q$ is $f(S\times 0)$. The mid-surface is only well-defined up to isotopy, because it depends on the choice
 of product structure.

The constants in the following definition are convenient but somewhat arbitrary.
  
 \begin{definition} \label{Def:ThinNiceFat}   
  
  A properly embedded geodesic $\alpha$ in a convex hyperbolic manifold $M$  \emph{has a fat collar} if any path $\gamma$ that starts and ends on $\alpha$ and has $\ell(\gamma) \leq 0.01$ is homotopic into $\alpha$.  In particular, $\alpha$ must be length--minimizing on a scale up to $0.01$. 

A \emph{fat tube} is a tube that contains a neighborhood of radius $r = 0.01$ about its core geodesic. In particular, the core of a fat tube has a fat collar.

A product structure on $Q \cong S\times I$ 
 is \emph{nice} if every vertical arc $\beta$  has length and curvature
  at most $0.01$, and if  the endpoints of $\beta$ meet $\bdy Q$ almost orthogonally. Here, ``almost orthogonal'' means that  unit tangent vectors to $\beta$ and $\bdy Q$ have inner product less than $0.01$.

 The  \emph{thickness} of a convex QF
manifold $Q$ with boundary $\partial_-Q\sqcup\partial_+Q$ 
is 
\[
t(Q)=\max \{d(x_-,\partial_+Q), \, d(x_+,\partial_-Q)  :\ x_{\pm}\in\bdy_{\pm} Q\}.
\]
\end{definition}

\begin{lemma}[Almost flat implies thin]\label{Lem:AlmostFlatThin}  For  every $\tau>0$, there is a constant $\kappa = \kappa(\tau) >0$, such that if 
$S$ is a QF surface with extrinsic curvature
everywhere less than $\kappa$, then $t(Q(S))<\tau$.
\end{lemma}

\begin{proof}
If $\kappa$ is small
enough, every geodesic in  the universal cover $\widetilde S$ is very close to a geodesic in $\HH^3$ with the same endpoints.
By Carath\'eodory \cite[Proposition 5.2.3]{Papadopoulos}, the convex hull of $\widetilde{S}$ 
is the union of tetrahedra with vertices in $\widetilde{S}$. Hence every point in $\widetilde{Q(S)}$ is very close to $\widetilde{S}$.
This implies $Q(S)$ is  thin.
\end{proof}

 \begin{lemma}\label{Lem:QFproduct} 
There exists $\tau > 0$ such that every convex QF manifold $Q$ with $t(Q) < \tau$ has a nice product structure.
 \end{lemma}
 
 \begin{proof} A  product structure can be constructed on $Q$ by 
 using a partition of unity to combine unit vector fields whose integral curves are geodesics and that are
defined in balls of radius $.001$ that cover $Q$ and are almost orthogonal to $\partial Q$. The flow
defined by this combined vector field can be reparameterized to give a nice product structure.  
 Further details are left to the reader.
  \end{proof}
 
  A $1$--manifold $\alpha \subset Q$ is called \emph{unknotted} if there is a product structure on $Q$ such that  $\pi_h \vert_\alpha$
 is injective.
Because any two product structures on $Q$ are isotopic,
 a closed curve  $\alpha \subset Q$  is unknotted if and only if it is isotopic to a simple closed curve in $\partial Q$.
An arc $\alpha$ properly embedded in $Q$ is unknotted  if and only if it properly isotopic
  to either a vertical arc, or to an arc in $\partial Q$. Geodesics in general QF manifolds can be knotted, but nice product structures preclude this  for geodesics with a fat collar:

  \begin{lemma}[Unknottedness]\label{Lem:Unknotted} 
  Suppose $Q$ is a QF manifold with a nice product structure. Let $\alpha \subset Q$ be a geodesic with a fat collar. Then $\alpha$ is compact and unknotted in $Q$.
 \end{lemma}

 \begin{proof} 
The fat collar about $\alpha$
prevents it from accumulating on itself inside $Q$ or traveling too deep into any cusps of $Q$. Thus $\alpha$ is compact.

 Let $F$ be the union of all vertical arcs (with respect to the nice product structure) that contain a point of $\alpha$. If some vertical arc $\beta$
 contains two points of $\alpha$, the niceness of the product structure and the fat collar about $\alpha$  imply that $\alpha$ is almost vertical. Thus $\alpha$ is isotopic to a vertical arc, hence unknotted. 
 Otherwise, if every vertical line in $F$ contains a single point of $\alpha$, then  $\pi_h \vert_\alpha$ is injective, hence $\alpha$ is unknotted by definition.
  \end{proof}

As a consequence, we have the following classification of geodesics in thin QF manifolds.

\begin{definition}\label{Def:SkirtMerLong}
Let $Q \cong S \times I$ be a QF manifold with a nice product structure, and let $\Th_{\infty}(Q)$ be the 
universal thickening
of $Q$.
Let $\delta \subset \Th_{\infty}(Q)$ be an embedded geodesic with a fat collar, and let $\alpha = \delta \cap Q$. 
We say that $\delta$ is
\begin{itemize}
\item \emph{skirting} if $\alpha = \emptyset$ or an interval whose endpoints are on the same component of $\bdy Q$.
\item \emph{meridional} if $\alpha$ is an interval whose endpoints are on different components of $\bdy Q$.
\item \emph{longitudinal} if $\alpha = \delta$ is a closed, unknotted geodesic in $Q$.
\end{itemize}

The terminology can be explained as follows. 
 In the meridional case,  $\bdy_+ Q \setminus \delta$ has a puncture that is a meridian of $\delta$. In the longitudinal case, $\delta$ is isotopic into $\bdy_+ Q$, hence removing it creates a pair of loops in  $\bdy_+ Q$ that are longitudes of $\delta$. 
 \end{definition}

\begin{lemma}\label{Lem:SkirtMerLong}
Let $R = \HH^3 / \Gamma \cong S \times \RR$ be a complete QF manifold, and let $\delta \subset R$ be an embedded geodesic with a fat collar. Suppose that some convex thickening $Q$  of $\core(S)$ has a nice product structure. Then $\alpha = \delta \cap Q$ is either empty or connected. Furthermore, $\delta$  is one of the three types enumerated in \refdef{SkirtMerLong}. The type of $\delta$ is independent of the choice of thickening of $\core(S)$.
\end{lemma}

\begin{proof}
Note that if $\alpha \neq \emptyset$, it must be connected by convexity of $Q$. In addition, $\alpha$ is compact and unknotted by \reflem{Unknotted}. Thus one of the above three possibilities must always hold.

Let $\widetilde \delta$ be an elevation of $\delta$ to $\widetilde R = \HH^3$. This is a bi-infinite geodesic with endpoints $x_\pm$. Let $\Lambda = \Lambda(\Gamma) \subset S^2_\infty$ be the limit set of $\Gamma$. Then $\delta$ is skirting if and only if $x_\pm$ lie in the same component of $S^2_\infty \setminus \Lambda$; meridional if and only if $x_\pm$ lie in different components of $S^2_\infty \setminus \Lambda$; and longitudinal if and only if $x_\pm \in  \Lambda$. This classification depends only on $\Gamma$, hence is independent of the choice of thickening of $\core(S)$.
\end{proof}

\subsection{Pancakes ensure ubiquity}\label{Sec:Pancake}

\begin{definition}\label{Def:Pancake}
For $0 < \eta < r$, a \emph{pancake} is 
\[ P(\eta, r ) = \neb_{\eta}(D_r ; \HH^3), \]
where $D_r$ is a closed disk of radius $r$ in a totally geodesic hyperbolic plane in $\HH^3$. The points of $\bdy P(\eta, r )$ that are distance $\eta$ from $\bdy D_r$ form an annulus called the  \emph{vertical boundary}, denoted $\bdy_V P(\eta, r )$. Meanwhile, $\bdy P(\eta, r ) \setminus \bdy_V P(\eta, r )$ consists of two disks, called the \emph{upper} and \emph{lower} boundary, and denoted $\bdy_\pm P(\eta, r )$. 

Let $X$ be a submanifold of $\HH^3$. We say that \emph{$X$ separates $P$} if $\bdy_-P$ and $\bdy_+ P$ are contained in different path-components of $P \setminus X$, and \emph{$X$ strongly separates $P$} if $\bdy_-P$ and $\bdy_+ P$ are contained in different path-components of $\HH^3 \setminus X$.
\end{definition}

\begin{lemma}\label{Lem:Pancake}
Let $\Pi_-, \Pi_+ \subset \HH^3$ be hyperbolic planes, such that $d(\Pi_-, \Pi_+) = 2\eta > 0$. Then there is a radius $r = r(\eta) > 0$ and a pancake $P = P(\eta, \, r(\eta))$ contained between $\Pi_-$ and $\Pi_+$, such that every convex submanifold $X \subset \HH^3$ that strongly separates $P$  also separates $\Pi_-$ from $\Pi_+$.
\end{lemma}

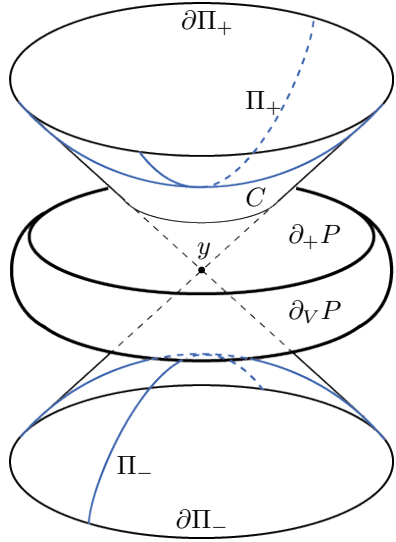
\begin{figure}
\begin{overpic}[scale=0.8]{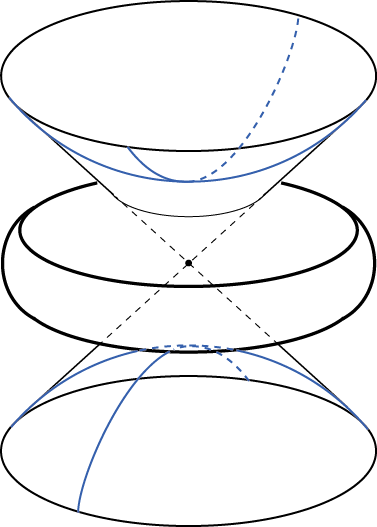}
\put(44,62){$C$}
\put(52,55){$\bdy_+ P$}
\put(52,41){$\bdy_V P$}
\put(44,80){$\Pi_+$}
\put(20,12){$\Pi_-$}
\put(32,95){$\bdy \Pi_+$}
\put(31,2){$\bdy \Pi_-$}
\put(35,53){$y$}
\end{overpic}
\caption{The construction of a pancake $P = P(\eta, r)$ in \reflem{Pancake}.}
\label{Fig:ConstructPancake}
\end{figure}

\begin{proof}
The pancake $P = P(\eta, \, r)$ is constructed as follows.
Let $\gamma$ be the unique geodesic segment of length $2 \eta$ connecting $\Pi_-$ to $\Pi_+$, and let $y$ be the midpoint of $\gamma$. Let $C \subset \HH^3$ be the double cone constructed by coning $y$ to $\bdy \Pi_\pm$.

The disk $D_r$  lies in the plane perpendicular to $\gamma$ at $y$. We choose a radius $r = r(\eta)$ large enough to ensure that the vertical boundary of $P(\eta, r ) = \neb_{\eta}(D_{r})$ lies entirely outside $C$. See \reffig{ConstructPancake}.

Let $X \subset \HH^3$ be a convex submanifold that strongly separates $P$. In particular, $X$ is disjoint from $\bdy_\pm P$.  Suppose, for a contradiction, that there is a point $x \in X \cap \Pi_+$. The geodesic line $\beta$ through $x,y$ must meet both $\bdy_-P$ and $\bdy_+ P$, hence contains points of $X$ between $\bdy_-P$ and $\bdy_+ P$. By convexity, the interval $\beta \cap X$  must intersect $\bdy_+ P$, since $x \in \beta \cap X$ is above $\bdy_+ P$. This contradicts the hypothesis that $X$ is disjoint from $\bdy_\pm P$.
 Thus $X$ must be disjoint from $\Pi_\pm$.

To show that $X$ strongly separates $\Pi_-$ from $\Pi_+$, consider a path $\beta \subset \HH^3$  from $\Pi_-$ to $\Pi_+$. Suppose, for a contradiction, that $\beta$ is disjoint from $X$. The geodesic segment $\alpha$ with the same endpoints as $\beta$ must 
intersect $P$ in an interval $I$, with $\bdy I \subset \bdy_\pm P$. Furthermore, $X \cap \alpha \subset P \cap \alpha = I$. By concatenating $\beta$ with the two components of $\alpha \setminus  \interior(I)$, we obtain 
a path from $\bdy_- P$ to $\bdy_+ P$ through $\HH^3 \setminus X$. This contradicts the hypothesis that $X$ strongly separates $P$.
\end{proof}

The point of \reflem{Pancake} is the following. If $M = \HH^3 / \Gamma$ is a cusped hyperbolic manifold, and $\Pi \subset \HH^3$ is a generic hyperbolic plane, the image of $\Pi$   is dense in $M$, and in particular  makes arbitrarily deep excursions into the cusps  \cite{Ratner:Topological, Shah:Closures}. When we Dehn fill $M$, we can control geometry using \refthm{BilipDrillFill}, but this control only works in regions of $M$ that stay out of the very thin  parts of the horocusps. Thus, in proving ubiquity, we cannot directly control what happens to $\Pi_\pm$  in Dehn fillings. By contrast,  the image of a pancake $P(\eta, \, r)$ in $M$ has bounded diameter, hence lies in $M^{\geq \epsilon}$ for some $\epsilon$. Thus we can use \refthm{BilipDrillFill} to understand what happens to the pancake during Dehn filling. This is used in the proof of \refthm{CuspedSurfacePancake}.

We will need to work with pancakes embedded in manifolds. If $f\from  \HH^3 \to M$ is a local isometry and $P \subset \HH^3$ is a pancake such that $f \vert_P$ is an embedding, we refer to the image $P_M = f(P) \subset M$ as a \emph{pancake in $M$}. The upper and lower boundary $\bdy_+ P_M$ and $\bdy_- P_M$ are well-defined via $f$. A submanifold $X \subset M$ is said to \emph{(strongly) separate $P_M$ in $M$} if there are elevations $\widetilde{X}$ of $X$ and $\widetilde{P}_M$ of $P_M$ to the universal cover $\widetilde{M}$, such that $\widetilde{X}$ (strongly) separates $\bdy_- \widetilde{P}_M$ from $\bdy_+ \widetilde{P}_M$. Since universal coverings are regular, one may first choose a lift $\widetilde{P}_M$ and then find an appropriate elevation $\widetilde{X}$.

  \section{Drilling Kahn--Markovic surfaces}\label{Sec:KMDrill}

The main goal of this section is to prove \refthm{CuspedSurfacesManySlopes}. That result  follows immediately from the following reformulated statement. 


\begin{theorem}\label{Thm:CuspedSurfacePancake} 
Let $M = \HH^3 / \Gamma$ be a cusped hyperbolic $3$--manifold. Let $\alpha_1, \ldots, \alpha_n$ be a collection of slopes  on cusps of $M$.
In addition, let $P = P(\eta, r)$ be a pancake in $\HH^3$. Then there are constants $k_i \in \NN$ and an immersed QF surface $F \to M$, such that the following hold:
\begin{enumerate}
\item\label{Itm:BoundarySlopes} For every $i$, at least one boundary component of $F$ is mapped to a multiple $k_i \alpha_i$.

\item\label{Itm:SepPancake} There is an elevation of $\core(F)$ to $\HH^3$ that strongly separates $P$.

\end{enumerate}
\end{theorem}

\begin{proof}[Proof of \refthm{CuspedSurfacesManySlopes}, assuming \refthm{CuspedSurfacePancake}]
For any pair of disjoint planes $\Pi_\pm$, whose distance is $2\eta$, construct a pancake $P = P(\eta, r)$ as in \reflem{Pancake}.
Let $F$ be a QF surface produced by  \refthm{CuspedSurfacePancake}. The elevation $\widetilde{\core}(F) \subset \HH^3$  that strongly separates $P$ is a convex set, hence $\widetilde{\core}(F)$ also separates $\Pi_-$ from $\Pi_+$. Thus the surfaces produced using \refthm{CuspedSurfacePancake} are ubiquitous.
\end{proof}

It is worth recalling that \refthm{CuspedSurfacePancake} and its consequence in  \refthm{CuspedSurfacesManySlopes} already implies a weak form of cubulation for $\pi_1(M)$. See \refcor{WeakCubulation} below for details.

The proof of \refthm{CuspedSurfacePancake} consists of two halves: filling and drilling. After passing to a cover where distinct slopes $\alpha_i$ lie on distinct cusps, we perform a long Dehn filling on $M$ along multiples of $\alpha_1, \ldots, \alpha_n$. The result is a closed hyperbolic $3$--orbifold $N$. For the sake of this outline, it helps to imagine that $N$ is a manifold. The Kahn--Markovic theorem gives a very thin immersed QF surface $S \to N$. We   study the intersection of a convex thickening $Q$ of $\core(S)$ with the union of Dehn filling tubes to build a \emph{convex envelope} $Z = \Qthick \cup \Vcal$, where $\Vcal$ consists of a subset of the tubes. This convex envelope is embedded in some cover of $N$; for now it helps to imagine that no cover is needed.

In the second half of the proof, we  surger the midsurface $S$ of $Q$, while taking care to stay within the convex envelope $Z$. We then drill out  the Dehn filling cores, recovering $M = N \setminus \Delta$. This produces a (possibly disconnected) surface $\Fcal \subset Z \setminus \Delta$. The convexity of $\bdy Z$ ensures that the components of $\Fcal$ have convex cores contained in $Z \setminus \Delta$, which implies they are quasi-Fuchsian.   

We begin the proof in \refsec{QFDrill} by laying out the drilling portion of the argument. See \refprop{QFDrilling} for a self-contained if somewhat lengthy statement. In \refsec{FillAndCover}, we lay out the Dehn filling argument, including repeated passage to covers, and incorporate \refprop{QFDrilling} to complete the proof of \refthm{CuspedSurfacePancake}.
In \refsec{MakingDo}, we sketch how the proof of \refthm{CuspedSurfacePancake} can be modified to avoid using several large hammers.

\subsection{Drilling a quasi-Fuchsian surface}\label{Sec:QFDrill}
We employ the following convention introduced by Baker and Cooper \cite{BCQFS}. From now until the start of \refsec{Cubulation}, calligraphic letters   always denote disjoint unions of objects (typically finitely many objects). The corresponding Roman letters denote the individual components. For instance, in the following proposition, $\Vcal$ 
denotes a disjoint union of tubes in a manifold $N$, whereas $V$ is a single tube forming a component of $\Vcal$. Similarly, $\delta(V)$ denotes the core of a tube $V$, while $\delta(\Vcal)$ is a disjoint union of all the cores.

\begin{proposition}\label{Prop:QFDrilling} 
Suppose that $N$ is a complete hyperbolic $3$--manifold and  $\Qthick \subset N$ is a compact, embedded QF submanifold with a nice product structure. Suppose that $\Vcal \subset N$ is a disjoint union of fat tubes, such that
every tube $V \subset \Vcal$ intersects $\Qthick$ in a single component of intersection and $\Qthick \cap \delta(V)$ is empty or has a fat collar.
Let $\Zunion = \Qthick \cup \Vcal$. Suppose $\Delta \subset N$ is a geodesic $1$--manifold such that 
$\Delta \cap \Zunion = \delta(\Vcal)$, the disjoint union of the cores of $\Vcal$.
  
Suppose that $M$  is a complete hyperbolic manifold with a diffeomorphism $\varphi\from  M \to N \setminus \Delta $, such that  $Y = \varphi^{-1}(\Zunion \setminus\Delta)$ is convex and has finite volume in the hyperbolic metric on $M$.

  Then   $M$ contains a (possibly disconnected) embedded surface $\Fcal$. Each component $F \subset \Fcal$ is quasi-Fuchsian, with $\core(F) \subset Y$. 
  The cusps of $\Fcal$ correspond to the components of $\Delta$ meeting $\Qthick$, as follows. 
For each meridional geodesic $\delta \subset \Delta$, one cusp of $\Fcal$  is a meridian of $\delta$. For each longitudinal geodesic $\delta \subset \Delta$, two cusps of $\Fcal$  are longitudes of $\delta$. Skirting geodesics do not contribute cusps of $\Fcal$.
  
 Finally, suppose that $P_M \subset M$ is an embedded pancake, such that $\Vcal \cap \varphi(P_M) = \emptyset$ and $\Qthick$ strongly separates $\varphi(P_M)$ in $N$. Then there is a component $F \subset \Fcal$ whose convex core strongly separates $P_M$. 
 \end{proposition}

In the statement of the proposition, $\varphi(P_M) \subset N$ is diffeomorphic but not necessarily isometric to a pancake. Since the definition of (strongly) separating a pancake is purely topological, the statement that $\Qthick$ strongly separates $\varphi(P_M)$ is unambiguous.

The surface $\Fcal$ in the statement of \refprop{QFDrilling} is constructed as follows. The tubes of  $\Vcal$ can be 
subdivided into three types, according to how their cores intersect $\Qthick$. (See \refdef{SkirtMerLong}.) For each type of tube, we   perform a local isotopy of the midsurface $S$ of $Q \cong S \times I$. After this isotopy, we let $\Fcal = S \setminus \Delta$. Most of the proof is devoted to verifying that $\Fcal$ has all the desired properties; this verification is broken up into a number of claims.

 \begin{proof}[Proof of \refprop{QFDrilling}]  
 Let $V$ be a tube component of $\Vcal$, let $\delta = \delta(V)$ be the core of $V$, and let $\alpha = \delta \cap \Qthick$. 
 If $H = \Qthick \cap V$ is non-empty, then it
is connected and convex, hence \reflem{ConvexProperties} implies
  $\pi_1 H$ is isomorphic to a subgroup
of $\pi_1 V \cong \ZZ$. 
 If $\pi_1 H \cong \ZZ$,
then $H$ is a (convex) tube, hence  $\delta \subset H \subset \Qthick$. Otherwise, $\pi_1 H \cong \{1\}$ and $H$ is a convex ball, hence $\alpha = \delta \cap H$ is 
empty or an arc.
In every case where $\alpha \neq \emptyset$, it is compact and unknotted in $\Qthick$ by \reflem{Unknotted}. 

The universal thickening $\Th_{\infty}(\Qthick)$ is a cover of $N$.
If $\alpha = \delta \cap Q$ is non-empty, the isometric inclusion $Q \hookrightarrow \Th_{\infty}(\Qthick)$ determines a unique complete  geodesic $\hat \delta \subset \Th_{\infty}(Q)$ extending $\alpha$ along a tangent vector. 
   Following \refdef{SkirtMerLong}, we call $V \subset N$ a \emph{skirting tube}, \emph{meridian tube}, or \emph{longitude tube} according to the type of $\alpha = \hat \delta \cap Q$.

   Let $S$ be the midsurface of $\Qthick$ with respect to the chosen nice product structure.
For each type of tube, we   perform an  isotopy of $S$ in $H$. 
This can be done independently for
each tube $V\subset \Vcal$, because the tubes in $\Vcal$ are compact and pairwise disjoint.

\smallskip
\noindent{\bf Skirting tube}:  
If $\alpha \neq \emptyset$, then it has both endpoints on the
same component of $\bdy \Qthick$.
Isotop $S$ inside $H$ so it is disjoint  from $\alpha$. If $\alpha = \emptyset$, no isotopy is needed.
In either case, $\Qthick \cup V$ is homeomorphic
to $\Qthick$ with $V$ glued onto the boundary along a disk.

There is a disk $\Sigma(V) \subset \Qthick \cap V$, with $\bdy \Sigma(V) \subset \bdy \Qthick \cap \bdy V$, and with
 $\alpha \cap \Sigma(V) = \emptyset$  that separates $\alpha$ from $S$. For future usage, we isotop each of $\Qthick$ and $V$ to its own side of $\Sigma(V)$. (After this isotopy, $\Qthick$ and $V$ may no longer be convex. However, the topological setup where $V$ is glued to $\Qthick$ along $\Sigma(V)$  is helpful below.)
 
\smallskip
\noindent{\bf Meridian tube}: 
In this case,  
 $H \cap \bdy \Qthick$  contains two disks $D_{\pm}(V) \subset\bdy_{\pm} \Qthick$, each containing an endpoint of $\alpha$,
and each serving as a compressing disk for $V$. Thus, after an isotopy of $S$ inside $H$,
we may assume $S \cap \delta$ consists of one transverse intersection point.
In this case, $\Qthick \cup V$ is homeomorphic to the manifold
 obtained by attaching a 1--handle to $\Qthick$ along $D_{\pm} (V)$. Define $\Sigma(V) = (D_-(V) \cup D_+(V)) \setminus \delta$.
 
\smallskip
\noindent{\bf Longitude tube}: 
In this case, $ V$ is a  fat tube that contains  the unknotted circle $\alpha=\delta(V)$. 
As in the proof of \reflem{Unknotted},
the vertical arcs in $\Qthick$ meeting $\alpha$ form
 an embedded annulus $A(V)$ that is contained in $V$ because $V$ is fat.  Thus $A(V)$ is
 properly
embedded in $H = \Qthick \cap V$ and contains $\alpha$. 
This annulus has one boundary component in
each component of $\bdy_{\pm}\Qthick$. After an isotopy of $S$ in $H$, we may assume that $S$ contains $\alpha$.

In the longitudinal case, $\Qthick \cup V$ is homeomorphic to $\Qthick$. For future usage, we take two parallel copies of $A(V)$, denoted $A_0(V)$ and $A_1(V)$, and isotop the solid torus $V$ inward, until $H = \Qthick \cap V$ is the portion of $\Qthick$ contained between $A_0(V)$ and $A_1(V)$. After this isotopy, $V$ or $H$ may no longer be convex, but $H$  still contains the closed geodesic $\delta$. Define $\Sigma(V) = A_0(V) \cup A_1(V)$.

\smallskip

After performing the above isotopy of $S$ for each component of $ \Vcal$, the midsurface $S \subset \Qthick$ has the following 
  properties. If $\delta$ is the core
 of a longitude tube, then $\delta\subset S$. If $\delta$ is the core of a meridian tube,
 then $\delta$ intersects $ S$ once transversely. If $\delta$ is the core of a skirting
 tube, then $\delta\cap S=\emptyset$.

Set $M=N\setminus\Delta$. Then $\Fcal = S\setminus\Delta$  is a surface properly embedded in $M$, which  is
 disconnected if and only if the union of the longitude tubes separates $S$.

Observe that $\Fcal$ has punctures of two kinds. Each meridian tube gives rise to one puncture of $\Fcal$, whose slope is a meridian of a component of $\Delta$. Each longitude tube gives rise to two punctures of $\Fcal$, whose slopes are 
longitudes of a component of $\Delta$. The skirting tubes do not contribute punctures.

From now on, we focus attention on a component $F \subset \Fcal$.
The following sequence of claims shows that $F$ has all the required properties.

\begin{claim}\label{Claim:Incompressible}
$F$ is incompressible in $M$. 
\end{claim}

Suppose otherwise. Then there is a compressing disk $D \subset M$, with $D\cap F = \bdy D = \alpha$, an essential simple closed curve on  $ F$. 
Since $F \subset S$ and $S$ is incompressible in $N$, there must be a disk
$D'\subset  S$ with $\bdy D'= \alpha$. This disk $D'\subset  S$ cannot contain any component of $\Delta$, because
each component is a geodesic in $N$.
Thus $D'$ meets each component of $\Delta$ transversely in at most one point.
Since $D \subset N \setminus \Delta$ is disjoint from $\Delta$ by hypothesis, the $2$--sphere $E = D \cup D'$ must meet each component  $\delta \subset \Delta$  transversely in at most one point. But the hyperbolic manifold $N$ is irreducible, so $E$ is separating, hence $E \cap \delta = \emptyset$. 
It follows that $D' = E \cap S$ is disjoint from $\Delta$, hence $D'\subset  F$, and so $\alpha$ is not essential in $ F$. This  contradiction proves the claim.

\smallskip

Recall that  $\Zunion=\Qthick \cup  \Vcal$.
 By hypothesis, $Y = \varphi^{-1}(\Zunion \setminus \Delta)$ is a convex submanifold of $M$.

\begin{claim}\label{Claim:NoAccidental}
$ F$ has no accidental parabolics in $M$. Consequently, the inclusion--induced representation $\rho\from  \pi_1 F \to \pi_1 M$ is faithful and type-preserving.
\end{claim}

Suppose, for a contradiction, that a non-peripheral loop in $ F$ is a parabolic in $M$. By Jaco's Annulus Theorem (\cite[Theorem VIII.13]{Jaco}; see also \cite[Lemma 2.1]{SSSSS}), there is an embedded annulus $B \subset M$ with one boundary component an essential loop $\alpha \subset  F$ and the other boundary component $\beta \subset T$, where $T \subset M$ is a horotorus. Since $\alpha \subset  F \subset Y$, and $Y$ is convex, $T$ must bound a neighborhood of some core curve $\delta$ of some tube of $\Vcal \subset \Zunion$.
Thus, after an isotopy, we may take $B$ and $T$ to be embedded in $Y$, with $\alpha \subset F \setminus \Vcal$. 

For each tube $V \subset \Vcal$, we have constructed a planar surface $\Sigma(V)$: this is a disk for skirting tubes, a pair of once-punctured disks  for a meridian tube, and a pair of annuli  for a longitude tube.  Define $\Sigma(\Vcal) = \bigcup_{V \subset \Vcal} \Sigma(V)$ to be the disjoint union of these planar surfaces. A small isotopy of $B$ ensures that it intersects  $\Sigma( \Vcal)$ transversely. Since $B$ and every component of $\Sigma(\Vcal)$ are incompressible in $Y$, a further isotopy ensures that each component of $B \cap \Sigma( \Vcal)$ is essential in both surfaces. No component of $B \cap \Sigma( \Vcal)$ can be an arc, because $\bdy \Sigma(V) \subset \bdy Y$ for each tube $V$, whereas $B$ is disjoint from $\bdy Y$. Thus $B \cap \Sigma( \Vcal)$
consists of simple closed curves $\gamma_1, \ldots, \gamma_k$ that are essential in both $B$ and $\Sigma(\Vcal)$. We order these curves from  $\alpha \subset \bdy B$ to $\beta \subset \bdy B$. Observe that $k >0$, because $\alpha \subset F \setminus \Vcal$ whereas  $\beta \subset \Vcal$.

It follows that
$B$ cannot enter a skirting tube  $V$, because for such a tube $\Sigma(V) = D(V)$ is a disk, which contains no essential closed curves. 

We now focus on $\gamma_1$, the curve of $B \cap \Sigma(\Vcal)$ that is closest to $\alpha$. Then $\gamma_1$ is the core curve of some annulus belonging to $\Sigma(V)$. If $V$ is a meridian tube, then $\gamma_1$ is the meridian of $\delta = \delta(V)$. Since $\alpha$ is isotopic to $\gamma_1$ through $B$, it follows that $\alpha$ is a meridian of $\delta$, hence peripheral in $F$.

If $V$ is a longitude tube, then $\gamma_1$ is isotopic to a longitude of $\delta = \delta(V)$. In addition, $\gamma_1$ is isotopic to $\alpha \subset F$ through $B$ and to the closed curve $F \cap \Sigma(V)$ through $\Sigma(V)$. Again, we conclude that $\alpha$ is peripheral in $F$. Thus, in all cases, $\alpha$ is not actually accidental.

\begin{claim}\label{Claim:QFinM}
$F$ is a quasi-Fuchsian surface in $M$, such that $\core(F) \subset Y = \varphi^{-1}(\Zunion \setminus \Delta)$.
\end{claim}

Recall that $Y \subset M$ is convex by hypothesis, and $F \subset Y$ by construction. Thus $\core(F) \subset Y$. 
 Since $\rho\from  \pi_1 F \to \pi_1 M$ is faithful and type-preserving by 
\refclaim{NoAccidental}, \refthm{BonahonCanary} says that $F$ is quasi-Fuchsian.

\begin{claim}\label{Claim:StrongSeparation}
Let $P_M \subset M$ be an embedded pancake, such that $\Vcal \cap \varphi(P_M) = \emptyset$ and
$\Qthick$ strongly separates $\varphi(P_M)$ in $N$. Let $F$ be the unique component of $ \Fcal = S \setminus \Delta $ that intersects $\varphi(P_M)$. Then each of $F$ and $\core(F)$ strongly separates $P_M$. 
\end{claim}

As in \refclaim{QFinM}, let $M_F$ be the cover of $M$ corresponding to $\pi_1 F$. Then $F$ has an isometric lift $\hat F \subset M_F$. This determines an isometric lift  $\hat P_M \subset M_F$, such that $\hat F \cap \hat P_M \neq \emptyset$.
We will show that the two components of $\bdy_\pm \hat P_M$ are contained in different path components of $M_F \setminus \hat F$, and in different path components of $M_F \setminus \core(\hat F)$. This  implies strong separation in the universal cover $\HH^3$.

Let $\hat \beta \subset M_F$ be a path from $\bdy_- \hat P_M$ to $\bdy_+ \hat P_M$. Projecting down to $M$, we obtain a path $\beta$ from $\bdy_-  P_M$ to $\bdy_+  P_M$. Mapping over to $N$, we obtain a path $\varphi(\beta)$ from $\varphi(\bdy_-  P_M)$ to $\varphi(\bdy_+  P_M)$. Since $\Qthick$ strongly separates $\varphi(P_M)$, we have $\Qthick \cap \varphi(\bdy_\pm P_M) = \emptyset$. In addition, $\Vcal \cap \varphi(\bdy_\pm P_M) = \emptyset$. Thus $\varphi(\beta)$ starts and ends outside $\Zunion = \Qthick \cup \Vcal$.

Recall that $\Qthick$ has an isometric lift $\hat \Qthick \subset R = \Th_\infty(Q)$, where $R \to N$ is the quasi-Fuchsian cover of $N$ corresponding to $Q$.
 The covering projection $M_F \to (R \setminus \hat \Delta)$ determines an image  $\hat \varphi(\hat P_M) \subset R$ that is an isometric lift of $ \varphi(P_M) \subset N$.
It also determines a unique image $\hat \varphi (\hat \beta)$, which runs from $\hat \varphi(\bdy_- P_M)$ to $\hat \varphi(\bdy_+ P_M)$. 
Since $\Qthick$ strongly separates $\varphi(P_M)$,  the path  $\hat \varphi (\hat \beta) \subset R$ must have an essential intersection with $\hat \Qthick$. Consequently, $\varphi(\beta)$ has an essential intersection with $\Qthick$, with the midsurface $S$ of $\Qthick$, and with
$ \Fcal = S \setminus \Delta$.

Observe that $\varphi( \beta)$ cannot have an essential intersection with any component surface of $\Fcal$ besides $F$. If such an intersection was to occur, the pullback of $\varphi(\beta)$ to $M_Y$ would start on $\bdy_- \hat P_M$ and end somewhere other than $\bdy_+ \hat P_M$, which means it would not agree with $\hat \beta$. Thus $\varphi(\beta)$ must have an essential intersection with $F$. 

This means $\hat \beta$ intersects $\hat F \subset M_F$. Since $\hat \beta$ was an arbitrary path in $M_F$ from $\bdy_- \hat P_M$ to $\bdy_+ \hat P_M$, it follows that the two components of $\bdy_\pm \hat P_M$ are contained in different path components of $M_F \setminus \hat F$. Hence $\hat F$ strongly separates $\hat P_M$, and $F$ strongly separates $P_M$.

To show that $\core(F)$ strongly separates $P_M$, it remains to check that $\core(F) \cap \bdy_\pm P_M = \emptyset$.
Recall that the endpoints of $\varphi(\beta)$ lie outside $\Zunion = \Qthick \cup \Vcal$. Thus the endpoints of $\beta$ lie outside $Y = \varphi^{-1} (\Zunion \setminus \Delta)$. On the other hand, the convex set $\core(F)$ is contained in $Y$. Thus $\beta$ starts and ends outside $\core(F)$. In the cover $M_F$, the path $\hat \beta$ starts and ends outside $\core( \hat F)$. Since $\hat \beta$ was an arbitrary path in $M_F$ from $\bdy_- \hat P_M$ to $\bdy_+ \hat P_M$, the conclusion follows.

This completes the proof of \refprop{QFDrilling}.
\end{proof}

\subsection{Dehn filling and covers}\label{Sec:FillAndCover}
Here are the main steps of the proof of \refthm{CuspedSurfacePancake}.
 
 \begin{enumerate}[1.]
 \item Replace $M$ by a finite cover where (elevations of) the slopes $\alpha_1, \ldots, \alpha_n$ lie on distinct cusps, and where the pancake $P$ embeds.
\smallskip
 
 \item 
 Perform a long Dehn filling, resulting in a hyperbolic orbifold $N = M(k \alpha_1, \ldots, k \alpha_n)$. The integer $k$ is chosen large enough so that \refthm{BilipDrillFill} preserves much of the geometry of $M$ in the filling. The cusps of $M$ get replaced by a union of tubes $\Wcal$. See properties (F1)--(F5) for details.
 
 \smallskip
 \item Map the pancake $P \subset \HH^3$ into $M$, and then into $N$ via the bilipschitz map $\varphi$. The result is a pancake $P_N$. Then, find a very large immersed disk $D \to N$ that  separates $P_N$, and that  has a transverse intersection with each core geodesic.
 See \refclaim{ErgodicDisk} for details.

 \smallskip
\item Apply the Kahn--Markovic theorem \cite{KahnM} to find an immersed QF surface $S \to M$ that closely fellow-travels the disk $D$. This surface  has transverse, meridional intersections with each core curve, and also strongly separates $P_N$.

 \smallskip
\item Pass to a cover $\hat N$ of $N$ where a certain thickening $Q_0$ of $\core(S)$ is embedded, and where $Q_0$ intersects each tube of $\hat \Wcal$ at most once. A finite cover with these properties exists by Agol's work \cite{Agol}; see also \refrem{Separability}.

 \smallskip
\item Working in the cover $\hat N$, apply \refthm{CombAsymmetric} to show that the union of $Q_0$ and the tubes that intersect it has a convex thickening $\Zunion$, such that $\bdy \Zunion$ closely fellow-travels  $\core(S)$ on the regions of interest. In particular, $\Zunion$ still separates the pancake $P_N$.
 
 \smallskip
 \item Apply \refprop{QFDrilling} to the convex envelope $\Zunion \subset \hat N$, recovering a disconnected QF surface $\Fcal$. This surface is embedded in the cover $\hat M$ of $M$ corresponding to the cover $\hat N$ of $N$. By construction, one component $F \subset \Fcal$  fellow-travels the disk $D$, hence  contains at least one meridional cusp projecting to a multiple of each $\alpha_i$. It  also separates the pancake.
  \end{enumerate}
 
\noindent The above outline deliberately omits any mention of quantitative constants that control thickness, embedded collars, and geodesic curvature. We now proceed to the full proof, with constants.

\begin{proof}[Proof of \refthm{CuspedSurfacePancake}]
Let $P = P(\eta,r)$ be a pancake in $\HH^3$.
Without loss of generality, assume that $5 \eta/4$ is smaller than the constant $\tau$ from \reflem{QFproduct}.
(Otherwise, make the pancake $P$ thinner and apply the same proof.) 
Let $\kappa = \kappa(\eta/4)$ be as in  \reflem{KappaConvex}. Then, for every convex set $X \subset \HH^3$,  the boundary $\bdy \Th_{\eta/4}(X)$ is $\kappa$--convex. In particular, the pancake $P = P(\eta, r) = \Th_{\eta/4}(\Th_{3\eta/4}(D_r))$ has $\kappa$--convex boundary.

By residual finiteness, there is a finite regular cover of $M$ where the pancake $P$ embeds, and where elevations of the peripheral curves $\alpha_1, \ldots, \alpha_n$ lie on distinct cusp tori. (See \cite[Lemma 2.1]{CooperLongReid} for an explicit recipe ensuring that a cusp torus $T \subset M$ has many preimages in a cover.) Observe that any QF surface in a cover $\overline M \to M$, satisfying the desired conditions \refitm{BoundarySlopes} and \refitm{SepPancake},   projects to a QF surface in $M$ with the same properties. Thus no generality is lost by assuming that every cusp  of $M$ contains exactly one slope  $\alpha_i$.

After the above reduction, we can consider the embedded pancake $P_M  \subset M$ that lifts to $P \subset \HH^3$. Thus $\bdy_+ P_M$ and $\bdy_- P_M$ are $\kappa$--convex, embedded surfaces in $M$.

Let $\epsilon < \mu_3$ be small enough to ensure that $P_M \subset M^{\geq \epsilon}$ and that $d( P_M,  \, \bdy M^{\geq \epsilon}) \geq 1$. 
Let $R = R(\eta/4)$ be as in \refthm{CombAsymmetric}. Then there is a constant $\epsilon' = e^{-(R+2)} \epsilon$ with the property that in every horocusp $C \subset M$, the $\epsilon'$--thin part of $C$ is at least distance $R+1$ away from $M^{\geq \epsilon}$.

Let $C_1, \ldots, C_n$ denote the horocusp components of $M^{\leq \epsilon}$, and let $C_i' \subset C_i$ be the corresponding horocusp components of $M^{\leq \epsilon'} \!$. By the above choice of $\epsilon'$, we have $d(\bdy C_i', \, \bdy C_i)  \geq R + 1 .$ The horotori $\bdy C_i$ and $\bdy C_i'$ are $1$--convex. 

Let $A = M \setminus \bigcup C_i$ and  $A' = M \setminus \bigcup C_i'$.
 Then we may think of $\alpha_i$ as a slope on either $\bdy C_i$ or $\bdy C_i'$, with the same normalized length. 
For a large integer $k$, let $N = M(k\alpha_1, \ldots, k \alpha_n)$ be a closed hyperbolic orbifold obtained by Dehn filling on $M$. More precisely, $k \geq 1$ needs to be large enough so that \refthm{BilipDrillFill} ensures a diffeomorphic embedding $\varphi\from  A' \to N$ with the following properties. 
\begin{enumerate}[(F1)]
\item If $S \subset A'$ is a $\kappa$--convex surface, then $\varphi(S)$ is $\kappa/2$--convex. If $\varphi(S) \subset \varphi(A')$ is a $\kappa$--convex surface, then $S$ is $\kappa/2$--convex.
\item The image $\varphi ( P_M)$ is a convex ball in $N$, with $\kappa/2$--convex boundary.
\item There is a thinner pancake $P_N = P(3\eta/4, r_N)$ isometrically embedded in $\varphi(A')$, so that $P_N$ separates $\varphi  (\bdy_- P_M)$ from $\varphi (\bdy_+ P_M)$.
\item For every $i \leq n$, the tori $\varphi(\bdy C_i) $ and $\varphi(\bdy C_i') $ are $\kappa/2$--convex. As a consequence, there is a nested pair of (convex) tubes $W_i' \subset W_i$, with $\bdy W_i' = \varphi(\bdy C_i')$ and $\bdy W_i = \varphi(\bdy C_i)$. 
\item\label{Itm:TubeRThicken} The lipschitz constants on $\varphi$ give $d(\bdy W_i', \, \bdy W_i) \geq R +  \eta$ and $d(\bdy W_i, \, \bdy P_N) \geq \eta$.
\end{enumerate}
Each of the above statements is ensured by conclusions \refitm{Bilip} and \refitm{CurvBound} of \refthm{BilipDrillFill}.

Let $\delta_i$ be the core of $W_i$. We define $\Wcal = W_1 \cup \ldots \cup W_n$ and $\Wcal' = W_1' \cup \ldots \cup W_n'$. Then each of $\Wcal$  and $\Wcal'$ is a disjoint union of convex tubes, with cores along $\Delta =  \delta_1 \cup \ldots \cup \delta_n$.

Consider a covering map $p \from \hat N \to N$. Let $\hat \Delta = p^{-1}(\Delta)$ and $\hat \Wcal = p^{-1}(\Wcal)$. This defines a corresponding cover $\hat M = \hat N \setminus  \hat \Delta$, whose hyperbolic metric is lifted from $M$.
The pancake $P_M \subset M$ has a lift $\hat P_M \subset \hat{M}$, namely the image of the pancake $P \subset \HH^3$ stipulated in the hypotheses. Then the diffeomorphic image $\hat \varphi( \hat P_M) \subset \hat N$ defines an isometric lift $\hat P_N \subset \hat N$, which still separates $\hat \varphi  \big( \bdy_- \hat P_M \big)$ from $\hat \varphi \big( \bdy_+ \hat P_M \big)$, as in condition (F3).
We call $\hat P_N \subset \hat{N}$ the \emph{preferred lift} of $P_N$ to $\hat N$.

Applying the above construction to $\widetilde{N} = \HH^3$ gives a preferred lift $\widetilde P_N \subset \HH^3$.

\begin{claim}\label{Claim:ErgodicDisk}
There is a totally geodesic disk $D \subset \HH^3$ with the following properties:
\begin{enumerate}[{\rm (D1)}]
\item The disk $D$ intersects at least one elevation of each $\delta_i \subset N$.
\item Every intersection of $D$ with an elevation of $\delta_i$ is transverse.
\item There is a number $\zeta > 0$ such that $ \neb_{\zeta + \eta/2}(D)$ separates the preferred lift $\widetilde P_N \subset \widetilde{N} = \HH^3$.
\end{enumerate}
\end{claim}

This follows as a corollary of a result by Shah \cite{Shah:Closures} and Ratner \cite{Ratner:Topological}: almost every geodesic plane $\Pi \subset \HH^3$ has dense image in $N$. Hence, an arbitrarily small perturbation of the midplane of $\widetilde P_N$  contains a disk with the desired properties.

Define a pancake $\widetilde P' = \neb_{\zeta}(D) \subset \HH^3$. Then, by (D3), the thicker pancake $\Th_{\eta/2} ( \widetilde P' ) =  \neb_{\zeta + \eta/2}(D)$ separates  $\widetilde P_N \subset \HH^3$. This thickened pancake $\Th_{\eta/2} ( \widetilde P' )$ maps to $N$ by a local isometry.

So far, $N$ is a closed hyperbolic orbifold, with singular locus of order $k$ along the core of every tube $W_i \subset \Wcal$. By Selberg's lemma,  a finite regular cover $\hat{N}$ of $N$ is a closed hyperbolic manifold. 
The preimage of $\Wcal$ in this cover, denoted $\hat \Wcal \subset \hat N$, is a disjoint union of tubes with non-singular core 
along the preimage $\hat \Delta$ of $\Delta$. 
Then $\hat{N}$ is a Dehn filling of $\hat{M} = \hat{N} \setminus \hat \Delta$. Note that the meridian of every component of $\hat \Delta$ is a primitive slope in $\hat{M}$, which maps to $k \alpha_i \subset M$ for some $i$.

 By residual finiteness, we may select $\hat N$ so 
that the locally isometric immersion $\Th_{\eta/2} ( \widetilde P' ) \to N$ lifts to an embedding in $\hat N$, with image an embedded pancake $\Th_{\eta/2} ( \hat P'_N )$.

We will pass to finite covers of $N$ several more times, keeping the name $\hat N$. Each newly constructed $\hat N$  is a cover of the previously constructed covers of $N$.
Each time we pass to a cover, we keep the complete preimage of the nested sets $\Delta \subset \Wcal' \subset \Wcal$, denoted $ \hat \Delta \subset \hat \Wcal' \subset \hat \Wcal$, respectively.
Each time, the manifold $\hat M = \hat N \setminus \hat \Delta $ is a finite cover of $M$. In contrast with the complete preimage $\hat \Wcal$, we keep \emph{only} the preferred lift of  the pancakes 
$P_M$ and $P_N$, 
denoted $\hat P_M$ and $\hat P_N$ respectively. 
Each time, the lifted diffeomorphism $\hat \varphi\from  \hat M \to \hat N \setminus \hat \Delta $  continues to satisfy  (F1)--(F5). 
Combining (F3) with (D3), we have an embedded pancake $\hat P'_N \subset \hat N$, with an embedded thickening $\Th_{\eta/2} (\hat P'_N)$ that separates $\hat \varphi  \big( \bdy_- \hat P_M \big)$ from $\hat \varphi \big( \bdy_+ \hat P_M \big)$.
See \reffig{NestedPancakes}.

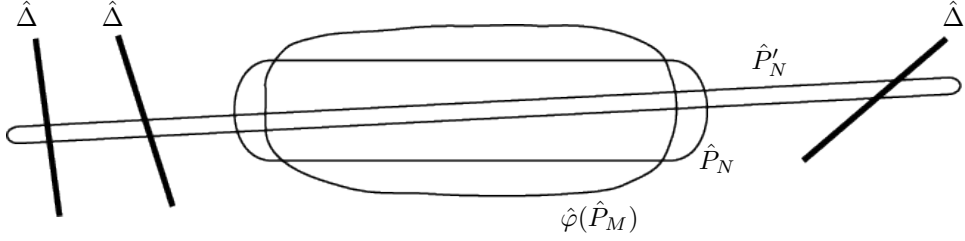
\begin{figure}
\vspace{0.05in}
\begin{overpic}[width=5in]{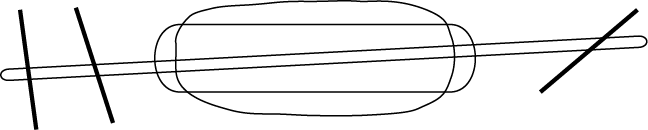}
\put(58,-1){$\hat \varphi(\hat P_M)$}
\put(72.5,5){$\hat P_N$}
\put(78,15.5){$\hat P'_N$}
\put(1,20){$\hat \Delta$}
\put(10,20){$\hat \Delta$}
\put(98,20){$\hat \Delta$}
\end{overpic}
\caption{A cross-sectional view of the nested pancakes appearing in the proof of \refthm{CuspedSurfacePancake}. All of these objects are embedded in a finite cover $\hat N$ of $N$.}
\label{Fig:NestedPancakes}
\end{figure}

By the Kahn--Markovic theorem \cite{KahnM}, there is a ubiquitous collection of immersed QF surfaces in $\hat N$, with arbitrarily small extrinsic curvature. 
    By \reflem{AlmostFlatThin}, a surface with small extrinsic curvature has a convex core  with very small thickness. 
Thus their theorem implies

\begin{claim}\label{Claim:KM}
There is a quasi-Fuchsian covering
 $\pi \from S \times \RR \to \hat N$, such that $\core(S)$ has a convex thickening $Q_0 = Q(S) \cong S \times I$ with the following properties.

\begin{enumerate}[{\rm (KM1)}]
    \item $t(Q_0) < \eta/4$. 

\item 
$\pi: Q_0 \to \hat N$ has an elevation  $\widetilde Q_0 \subset \widetilde{N} = \HH^3$, which strongly separates the preferred lift $\widetilde P'_N$.

\item Every component of $\pi^{-1} (\hat \Delta)$ that intersects $\pi^{-1} (\hat P'_N)$ is meridional in $Q_0$
 (see \refdef{SkirtMerLong}). 
\end{enumerate}
\end{claim}

Property (KM1) holds by \reflem{AlmostFlatThin}, because Kahn--Markovic surfaces are almost geodesic. Property (KM2) holds because these surfaces are ubiquitous. Property (KM3) holds because $S$ can be chosen so that an elevation $\widetilde Q_0 = \widetilde Q(S)$ lies arbitrarily close to the disk $D$ of \refclaim{ErgodicDisk}, whose intersections with the elevations of $\Delta$ are transverse.

By Agol's theorem \cite[Theorem 9.2]{Agol}, $\pi_1(\hat N)$ is subgroup separable, 
hence we may pass to a finite cover of $\hat N$ in which $\Th_8 (Q_0)$ is an embedded submanifold. We replace $\hat N$ by this cover, retaining the name $\hat N$. As described above, we keep the complete preimage of the unions of tubes $\Wcal$ and $\Wcal'$, and the single preferred lift of each pancake.

Applying subgroup separability again, we pass to a finite cover of $\hat N$ 
such that $\Th_\eta (Q_0)$ lifts to the cover, and $\hat W \cap \Th_\eta ( Q_0)$ is empty or connected for every component $\hat W \subset \hat \Wcal$.
 By convexity, this means $\hat W' \cap \Th_\eta ( Q_0)$ is empty or connected for every component $\hat W' \subset \hat \Wcal'$. 
 
Let $\hat W_1', \ldots, \hat W_m'$ be the components of $\hat \Wcal'$ with the property that $\Th_{\eta/2}(\hat W_j') \cap Q_0 \neq \emptyset$. 
Let $\hat W_1, \ldots, \hat W_m$ be the corresponding components of $\hat \Wcal$.  
We apply \refthm{CombAsymmetric} to $Q_0$ and these tubes. That is, $Q_0$ and $\Th_8(Q_0)$ play the roles of $M_0$ and $Y_0$, respectively. For $1 \leq j \leq m$, the nested tubes $\Th_{\eta/2} (\hat W_j')$ and $\hat W_j$  play the roles of $M_j$ and $Y_j$, respectively. Note that each of the above submanifolds of $\hat N$ is convex. The tubes $\hat W_j$ are disjointly embedded by the Dehn filling construction of $N$. Furthermore, condition (F5) gives $\hat W_j \supset \Th_{R+ \eta}(\hat W_j')$. Thus all the hypotheses of \refthm{CombAsymmetric} hold for $R = R(\eta/4)$, and we have
\begin{align*}
\Zunion
& := \Th_{\eta/4} \CH \left( Q_0  \cup \Th_{\eta/2} ( \hat  W_1') \cup  \ldots \cup \Th_{\eta/2} (\hat  W_m' ) \right) \\
& \subset \Th_{\eta/4} \left( \Th_{\eta/4} (Q_0) \cup \Th_{R+\eta/2}(\hat W_1') \cup  \ldots \cup \Th_{R+\eta/2}(\hat W_m') \right) \\
& = \Th_{\eta/2} (Q_0) \cup \Th_{R+ 3\eta/4} (\hat W_1') \cup  \ldots \cup \Th_{R+ 3\eta/4} (\hat W_m')  \\
& \subset    \Th_{\eta/2} (Q_0) \cup \hat W_1 \cup  \ldots \cup \hat W_m .
\end{align*}
Here, the first containment is by  \refthm{CombAsymmetric} and the second containment is by condition (F5).

Define
$  \Qthick =  \Zunion \cap \Th_{\eta/2} (Q_0) $
  and
$\Vcal = \Zunion \cap (\hat W_1 \cup \ldots \cup \hat W_m)$. Then $\Vcal$ is a disjoint union of tubes, each of which has connected intersection with $\Qthick$.
 Since $t(Q_0) < \eta/ 4$, and the thickening adds $\eta/2$ to each side, we have $t(\Qthick) < 5 \eta / 4$. Thus, by the choice of $\eta$ at the beginning of the proof,  \reflem{QFproduct} implies $\Qthick$ has a nice product structure.

Recall that $\Zunion$ is the $(\eta/4)$ thickening of a convex set in $\hat N$.
Thus, by the definition of $\kappa$ at the beginning of the proof, $\bdy \Zunion = \bdy ( \Qthick \cup \Vcal)$ is $\kappa$--convex. 
By construction, 
every component of $\hat \Wcal' \cap \Vcal$ is properly contained \emph{inside} $\Zunion$. By the definition of $ \hat W_1', \ldots, \hat W_m' \subset \Vcal$, every tube of $\hat \Wcal' \setminus \Vcal$ is disjoint from $\Th_{\eta/2}(Q_0)$, which means that these tubes lie entirely \emph{outside} $\Zunion$.
Thus $\bdy \Zunion \subset \hat N \setminus \hat \Wcal' = \hat \varphi (\hat A')$,  where $\hat A' \subset \hat M$ is the region outside the horocusps on which $\hat \varphi$ has the desired metric properties. Therefore, by property (F1), it follows that $\hat \varphi^{-1} (\bdy \Zunion) \subset \hat M$ is $\kappa/2$--convex.

 Let $Y = \hat \varphi^{-1}(\Zunion \setminus \hat \Delta) \subset \hat M$. Then the (finitely many) tubes of $\hat \Wcal'$ that lie inside $\Zunion$ are replaced by (finite volume) horocusps in $Y$. Thus $Y$ has finite volume. Furthermore, 
  $\bdy Y = \hat \varphi^{-1}(\bdy \Zunion)$ is $\kappa/2$--convex in the hyperbolic metric on $\hat M$. Then an elevation of $Y$ to $\widetilde M$ is a
 closed, connected subset of $\HH^3$ with locally convex boundary, which means it is convex. Thus $Y$ is convex as well.

We have now checked that the submanifolds $ \Qthick, \Vcal, \hat \Delta \subset \hat N $ and $Y \subset \hat M$
 satisfy all the hypotheses required by \refprop{QFDrilling}. That proposition performs an isotopy of $S$ inside $\Qthick$, after which every component of $\Fcal = S \setminus \hat \Delta$  is quasi-Fuchsian in the metric on  $\hat M$. To complete the proof of the theorem, we claim that one component $F \subset \Fcal$ has all the desired properties.

\begin{claim}\label{Claim:CuspedSurfaceProps}
There is a unique component  $F \subset \Fcal$ such that $F \cap \hat P_M \neq \emptyset$.  Furthermore,
\begin{enumerate}
\item For every slope $\alpha_i \subset M$, at least one cusp of $ F$ maps to some multiple $k_i \alpha_i$.

\item There is an elevation of $\core(F)$ to $\HH^3$ that strongly separates $P = \widetilde P_M$.
\end{enumerate}

\end{claim}

Recall that by property (KM2), $Q_0$ strongly separates $\hat P'_N$, hence $Q \subset \Th_{\eta/2}(Q_0)$ strongly separates $\neb_{\eta/2} \big( \hat P'_N \big) $. Combining (F3) with (D3), as above, we conclude that $\Qthick$ strongly separates  $ \hat \varphi \big( \hat P_M \big)$. In particular, $\Qthick \cap  \hat \varphi \big( \hat P_M \big) \neq \emptyset$ and  $Q_0 \cap  \hat \varphi \big( \hat P_M \big) \neq \emptyset$.  See \reffig{NestedPancakes}.

By property (KM3), every component of $\hat \Delta$ that intersects $\hat P'_N$ is meridional in $Q_0$. By \reflem{SkirtMerLong}, every component of $\hat \Delta$ that intersects $\hat P'_N$ is also meridional in $\Qthick$. In particular, $\hat \Delta \cap \hat P'_N$ does not disconnect the midsurface $S$ of $\Qthick$. Thus there is exactly one component $F \subset \Fcal$ such that $F \cap \hat P_M \neq \emptyset$. 

By \refprop{QFDrilling}, every component of $\hat \Delta$ that is meridional in $\Qthick$ gives rise to a cusp of $\Fcal$ that is a meridian of $\Delta$. By property (D2), the components of $\hat \Delta \cap \hat P'_N$ include geodesics that project to every core $\delta_i \subset N$. Thus, for every slope $\alpha_i$ on $M$, the component $F \subset \Fcal$ contains at least one cusp projecting to $k_i \alpha_i$. (The multiple $k_i$ incorporates the regular cover of $M$ constructed at the very beginning of the proof, as well as the manifold cover of the orbifold $N$ constructed using Selberg's lemma.)

Finally, since $\Qthick$ strongly separates  $ \hat \varphi \big( \hat P_M \big)$, and $\hat \Wcal \cap \hat \varphi \big( \hat P_M \big) = \emptyset$ by property (F5),  \refprop{QFDrilling} implies that $\core(F)$ strongly separates $P_M$. Thus some elevation of  $\core(F)$ to $\HH^3$ strongly separates $P = \widetilde P_M$.
\end{proof}

\subsection{Making do with less}\label{Sec:MakingDo}
In this section, we outline how the proof of
 \refthm{CuspedSurfacePancake} can be modified to avoid appealing to Shah's theorem \cite{Shah:Closures}, Agol's theorem \cite{Agol}, or the work of Brock and Bromberg \cite{brock-bromberg:density}.
The upshot of Remarks~\ref{Rem:NoBB}--\ref{Rem:NoErgodicity} is that almost all of the technical work in the proof of \refthm{CuspedSurfacesManySlopes} can be handled by soft classical arguments. The one ``major hammer'' needed for the proof is the Kahn--Markovic theorem \cite{KahnM}.

\begin{remark}\label{Rem:NoBB}
In the proof of \refthm{CuspedSurfacePancake},
our use of 
    Brock and Bromberg's 
\refthm{BilipDrillFill} can be replaced by a softer appeal to Thurston's results on geometric convergence under Dehn filling.
Suppose $M$ is a finite volume hyperbolic manifold and $X \subset M$ is a compact submanifold obtained by removing a set of horocusps. 
Then the map $\devGX(X)\rightarrow\Hom(\pi_1 X, \, \PSL(2,\CC))$ that sends the developing map for an incomplete hyperbolic metric on 
$X$ to its holonomy is open. See Thurston  \cite[Theorem 5.8.2]{WPT}, Goldman
 \cite{GOLD}, and Choi \cite{CHOIGS}; compare Cooper, Long, and Tillmann \cite[Proposition 1.2]{CLT3}. It follows that for any sufficiently large Dehn filling $N$ of all the cusps of $M$,
there is a diffeomorphic embedding of $X$ into $N$ that is close in the smooth topology to an isometry. In particular, such a map is $(1+\epsilon)$--bilipschitz
and the second derivative is close to $0$. The control of second derivatives ensures that $\kappa$--convex surfaces stay convex.
\end{remark}

\begin{remark}\label{Rem:Separability}
In the proof of \refthm{CuspedSurfacePancake},
our appeal to Agol's theorem on subgroup separability \cite{Agol} is convenient but not actually necessary.

Picking up the proof 
after \refclaim{KM}, 
suppose we have found
 an immersed QF surface $S \to \hat N$, with immersed convex thickening $Q_0$.
Then  
we may pass to an infinite cover $\overline{N} \to \hat N$, where $\Th_8(Q_0)$  embeds and $\hat \Wcal \subset \hat N$ pulls back to a non-compact submanifold $\overline \Wcal$ with the following properties.
For every component $\overline W \subset \overline \Wcal$ such that $\Th_\eta (Q_0) \cap \overline W \neq \emptyset$, the intersection must be connected and $\overline W$ must be a compact tube. 
For concreteness, one could take $\overline N$ to be the cover corresponding to $\pi_1(Q_0)$ amalgamated with one extra loop (about some power of a core) for every skirting or meridional intersection $\Th_\eta (Q_0) \cap \hat \Wcal$. 
The compactness of $Q_0$ ensures that only finitely many tubes intersect $\Th_\eta (Q_0)$, hence  $\pi_1( \overline N)$ is finitely generated. In fact, the virtual amalgamation theorem of Baker and Cooper \cite[Theorem 5.3]{BC1} ensures that $\overline N$ is geometrically finite.

With this setup, \refthm{CombAsymmetric} applies exactly as above to give a  compact, convex manifold $\Zunion = \Qthick \cup \Vcal$.  Let
$\overline \Delta$ be the union of \emph{all} the cores of $\overline \Wcal$, which may have non-compact components and infinitely many components. Then
\refprop{QFDrilling} applies to give an embedded QF surface in $\overline{M} = \overline{N} \setminus \overline{\Delta}$.  Since $\Zunion \cap \overline \Delta$ consists of  finitely many closed geodesics, it still the case that $Y = \varphi^{-1} (Z \setminus \overline \Delta)$ has finite volume.  Note that $\overline{M}$ inherits its hyperbolic metric from $M$, and that \refprop{QFDrilling} does not require any manifold except $Y$ to have finite volume. The surface  $F \subset \overline M$ projects to an immersed QF surface in $M$, and \refclaim{CuspedSurfaceProps} verifies that $F$ has all the desired properties.
%
\end{remark}

\begin{remark}\label{Rem:NoErgodicity}
In the proof of \refthm{CuspedSurfacePancake}, we use \refclaim{ErgodicDisk} (which follows from the work of Shah \cite{Shah:Closures} and Ratner \cite{Ratner:Topological})
to ensure that a certain region of a Kahn--Markovic surface hits all the Dehn filling cores transversely. This is used in \refprop{QFDrilling} to produce a single component $F \subset \Fcal$ with cusps along some multiple of \emph{every} slope $\alpha_i$. However, even without \refclaim{ErgodicDisk}, we could focus on \emph{one} slope $\alpha_i$ and apply \refprop{QFDrilling} to find a connected QF surface $F_i$ with at least one cusp mapping to a multiple of $\alpha_i$. 
 Moreover, the set of surfaces with this property is ubiquitous. 
 
Given surfaces $F_1, \ldots, F_n$ realizing slopes $\alpha_1, \ldots, \alpha_n$, one could use the techniques of Baker and Cooper \cite{BC1} to build a single surface 
$F$ immersed in $M$, such that every cusp of $F_i$ is covered by a cusp in $F$. As we explain in \refsec{BCGluing} below, such a surface $F$ can be obtained by gluing together subsurfaces of finite covers of the $F_i$. By the argument in Sections~\ref{Sec:ClosedUbiquitous} and \ref{Sec:UbiquitousSlope}, the set of such surfaces is ubiquitous. This line or argument recovers \refthm{CuspedSurfacePancake} and proves Theorems~\ref{Thm:CuspedSurfacesOneSlope} and~\ref{Thm:CuspedSurfacesManySlopes} without relying on \refclaim{ErgodicDisk}.
\end{remark}

\section{Gluing and prefabrication}\label{Sec:BCGluing}

This section summarizes without proof some of the results of Baker and Cooper \cite{BCQFS}.
In what follows, $M$ is a complete finite volume hyperbolic $3$--manifold
and $Q$ is a finite volume QF manifold.

When working with $3$--manifolds that contain surfaces with
 cusps, it is convenient to isotop everything so the cusps have compatible product structures.
%
Suppose $ B\subset{\mathbb H}^n$ is a horoball  centered on a point $x\in\bdy \HH^n$ bounded by the horosphere
 $ H= \bdy B$. A \emph{vertical ray} is a ray in $ B$ that starts on $ H$ and limits on $x$. Given
$P\subset H$, \emph{the set lying above $P$} is called a \emph{vertical set} and
  is the union, $V(P)$, of the vertical rays starting on $P$. 
If $P \subset  H$ is  convex, $V(P)$ is called a \emph{thorn} and 
$P$ is called the \emph{base of the thorn}. 
If $P=I\times \RR$ is an infinite strip,  $V(P)$ is a \emph{slab}.

A hyperbolic $n$--manifold $E$ is an \emph{excellent end} if it has finite volume and is isometric to $V/\Gamma$ for some
vertical set $V\subset B$ and discrete group $\Gamma\subset \Isom( \HH^n )$ preserving $V$. 
The \emph{horospherical boundary} of $E$ is $ \bdy_{ H} E=(V\cap  H)/\Gamma$.
An \emph{excellent rank--$1$ cusp} is a $3$--manifold $V/\Gamma$ where $V$ is a slab and $\Gamma$ is a cyclic group of parabolics
preserving $V$.

A (possibly disconnected) hyperbolic manifold $M$ is \emph{excellent} if  $M=M^c\cup\Vcal_M$ where $M^c$ is
compact, each component of $\Vcal_M$ is an excellent end,  and $M^c\cap\Vcal_M=  \bdy_{ H} \Vcal_M$.
The pair $(M^c,\Vcal_M)$ is called an \emph{excellent decomposition} of $M$. 
For example, an ideal convex polytope is  an excellent manifold whose ends are thorns.
 In addition, a complete hyperbolic manifold with finite volume is excellent, since any ends of $M$ are horocusps.
     It is routine to show that if $S$ is a QF surface then $Q(S)$ has ends that
     are excellent rank--1 cusps, hence $Q(S)$ is excellent.
     
      A compact, orientable surface properly embedded in a compact orientable $3$--manifold is \emph{essential}  if  
it is  incompressible and $ \bdy$--incompressible. 

\begin{definition}\label{Def:NiceSurface} A  surface $S$  embedded
 in an excellent $3$--manifold  $M=M^c\cup\Vcal$ is \emph{excellently essential} if each component of
 $S\cap\Vcal$ is an excellent annulus in the induced metric on $S$, and 
$S^c=S\cap M^c$ is a compact essential surface in $M^c$ with $ \bdy S^c\subset  \bdy M^c$. 
\end{definition}

For example, if $Q(S)$ is embedded in $M$, then both components
 of $ \bdy Q(S)$ are excellently essential. Recall that a \emph{slope} on a torus is an isotopy class of essential simple closed curves. In view of the preceding, it makes
sense to talk about the \emph{slope} of a excellently essential surface $S$ in a cusp of $M$, and the slope of a rank--$1$
cusp embedded in a rank--$2$ cusp.

\begin{definition}\label{Def:IdealSpider}  An \emph{ideal spider}
  is an excellent convex hyperbolic $3$--manifold $X$ with simply connected ends. Thus there is
  an excellent decomposition
  $X=B\cup {\Lcal}$ such that $B$ is compact and convex, and each component of $\Lcal$ is
 a thorn. The components of $\Lcal$ are called  \emph{legs} and  $B$ is called the \emph{body}.
  \end{definition}

The definition implies that the holonomy of an ideal spider
 has no parabolics. A special case of an ideal spider  with $k$ legs is a
 convex polytope with $k$ ideal vertices. In general, the body of a spider need not be simply connected. 
  The following is immediate:

\begin{proposition}[Proposition 3.3 in \cite{BCQFS}]\label{Prop:3SpiderLemma} Suppose $M$ is a complete hyperbolic $3$--manifold 
 and $Q_1,Q_2\subset M$ are excellent  QF
submanifolds. Then $Q_1\cap Q_2$ is excellent.
 If $Q_1$ and $Q_2$ have different slopes in every cusp of $M$, then each component of $Q_1\cap Q_2$
is an ideal spider. \qed
\end{proposition}

If $M$ and $N$ are excellent hyperbolic manifolds, a map $f\from M \to N$ is \emph{excellent}
if it is a local isometry and there are excellent decompositions with $f^{-1}(N^c)=M^c$. 
It follows that each vertical ray in $\Vcal_M$ maps to a vertical ray in $\Vcal_N$.

An \emph{immersed QF manifold} is a triple $(M,Q,f)$, where $f\from Q \to M$ is an excellent map between excellent hyperbolic
$3$--manifolds and $Q$ is quasi-Fuchsian. 
An  \emph{immersed ideal spider} is $(M,R,p)$ where $M$ is an excellent $3$--manifold, $R$ is an ideal spider,
 and $p\from R \to M$ is an excellent map.

 If $N$ is a submanifold of a cover of a hyperbolic manifold $M$,
  the restriction of the covering  projection
 gives a local isometry $p\from N \to M$,  called the \emph{natural projection}.
  If $S$ is an immersed QF surface in $M$,  the natural projection
 $Q(S) \to M$ is excellent.
  The following result generalizes \refprop{3SpiderLemma} to immersed QF manifolds.
    
     \begin{theorem}[Theorem 3.4 in \cite{BCQFS}]
     \label{Thm:GluingRegion} 
   Suppose that $(M,Q_1,f_1)$  and $(M,Q_2,f_2)$ are immersed QF manifolds.  Whenever rank--$1$ cusps $V_1 \subset Q_1$ and $V_2 \subset Q_2$ map to the same  cusp of $C \subset M$, suppose  $f_1(V_1)$ and $f_2(V_2)$  represent linearly independent homology classes in $H_1(C)$.
  Suppose there are basepoints $q_i\in Q_i$, such that the basepoint  $m=f_1(q_1)=f_2(q_2)$ is located in an excellent end of $M$.
   
   Then there is a connected hyperbolic $3$--manifold $P= \hat{Q}_1\cup \hat{Q}_2$, where $p_i\from  \hat{Q}_i \to Q_i$ is a finite covering and $R= \hat{Q}_1\cap \hat{Q}_2$ is an ideal spider with at least 2 legs.
Thus $(Q_i,R,p_i|_R)$ is an immersed ideal spider. 
   
   The holonomy provides an identification of $\pi_1(M,m)$ with a Kleinian 
group $\Gamma\subset \Isom({\mathbb H}^3)$. Then the QF manifolds
$Q_i$ have holonomy
 $\Gamma_i=(f_i)_*(\pi_1(Q_i,q_i))\subset\Gamma$ and the holonomy of $R$ is $\Gamma_1\cap\Gamma_2$.
      \end{theorem}

Part of the point of \refprop{3SpiderLemma} and \refthm{GluingRegion} is that
 the spider $R = \hat Q_1 \cap \hat Q_2$ retains all the information
  for gluing $\hat Q_1$ to $\hat Q_2$. More evocatively, \emph{spiders know how to sew a web}. This  is useful in our constructions, where we
   want to glue further covers of $\hat Q_1$ and $\hat Q_2$. If the spiders lift to the covers, they retain all the necessary gluing instructions.

The goal in our proof of \refthm{ClosedQFSurfaces} is to take several immersed QF submanifolds of $M$ and some rank--2 cusps of $M$, and
glue copies of finite covers of these manifolds to create
a geometrically finite hyperbolic manifold $\prefab$ that has a convex thickening, such that $ \bdy \prefab$
consists of closed QF surfaces.

\begin{definition}[Definition 1.1 in \cite{BCQFS}]\label{Def:Prefab}
A \emph{prefabricated manifold} is a connected, metrically complete, finite-volume, hyperbolic $3$--manifold 
$$\prefab=\Ccal\cup \Qcal_1\cup \Qcal_2 .$$
Each component of $\Qcal_i$ and of $\Ccal$ is a convex hyperbolic $3$--manifold called a \emph{piece}. 
Each component of $\Qcal_i$ is a 
quasi-Fuchsian $3$--manifold with at least one cusp. 
Each component of $\Ccal$ is  a horocusp. These pieces satisfy the 
following conditions for $i\in\{1,2\}$, and for each component $C$ of $\Ccal$:
\begin{enumerate}[(P1)]
\item $\Qcal_i\cap\Ccal$ is the disjoint union of all the cusps in $\Qcal_i$,
\item  $\Qcal_i\cap  \bdy C$  is an annulus with core curve $\alpha_i(C)$,
\item $\alpha_1(C)$ intersects  $\alpha_2(C)$ once transversely,
\item Each component of $\Qcal_1\cap \Qcal_2$ intersects $\Ccal$.
\end{enumerate}
Note that by \refprop{3SpiderLemma}, each component of $\Qcal_1\cap \Qcal_2$ is an ideal spider. 
\end{definition}

If pieces of a prefrabricated manifold $\prefab$ are sufficiently far apart, it  has a convex thickening. 

\begin{corollary}[Corollary 1.4 in \cite{BCQFS}]\label{Cor:PrefabConvex} Suppose $r\ge 8k$ where $k=(|\Ccal|+|\Qcal_1|+|\Qcal_2|-1)$,  and suppose
\begin{enumerate}[{\rm (Z1)}]
\item $\prefab^{r}={\mathcal C}^{r}\cup \Qcal_1^{r}\cup\Qcal_2^{r}$  is a prefabricated manifold, 
\item  $\prefab={\mathcal C}\cup \Qcal_1\cup \Qcal_2$ is a prefabricated manifold contained in $\prefab^{r}$,
\item $\Qcal_i^{r}$ is a thickening of $\Qcal_i$,
\item $\Ccal^{r}=\Th_{r}(\Ccal)$,
\item $\Qcal_i^{r}$ contains an $r$-neighborhood of $\Qcal_i\setminus \Ccal$,
\item Every component of $\Qcal_1^{r}\cap \Qcal_2^{r}$ contains a point of $\Qcal_1\cap \Qcal_2$.
\end{enumerate} Then $\prefab$ has a convex thickening that is a submanifold of $\prefab^{r}$. Moreover,
every component of $ \bdy \prefab$ is compact and quasi-Fuchsian.
\end{corollary}

To achieve the metric separation needed for \refcor{PrefabConvex}, one starts with a prefabricated
manifold $\prefab$ and replaces it by another prefabricated manifold $\prefab'$, constructed  from large finite covers of the pieces
of $\prefab$ with the property
that all spiders used for the gluing lift. (Compare \refsec{Elevation}.) Choices of lifts of the spiders then determine $\prefab'$. 
If the cover of each piece is regular (e.g.\ cyclic) and large enough, then one can choose lifts that are far apart.

Since a QF manifold is the product of a surface and an interval, constructing such covers
reduces to questions about coverings of surfaces that contain various immersed surfaces (corresponding to the spiders).
The existence of the right coverings of surfaces is shown in \cite[Theorem 2.8]{BCQFS}. 
 The argument makes heavy use of subgroup separability arguments in surface groups \cite{Scott:LERF}.
 
  A covering space $p\from \hat F\rightarrow F$ is called 
 {\em conservative} if the surfaces $F$ and $\hat{F}$ have the same number of boundary components. 
  We frequently
 wish to take large conservative covers of surfaces 
 with the property that certain immersed spiders lift to embedded spiders that are far apart. This is done using
 
 \begin{theorem}[Theorem 0.1 of \cite{BC2}; compare Theorem 9.1 of  \cite{MZ1}] \label{Thm:Conservative}
 Let $F$ be a compact, connected surface with 
$\bdy F\ne\emptyset$  and $H\subset\pi_1F$  a finitely generated subgroup. Assume that no loop representing
an element of 
$H$ is freely homotopic into $\bdy F$. Given a  finite subset $B\subset\pi_1F\setminus H$, there exists a finite-sheeted conservative cover $p\from \hat{F}\rightarrow F$ 
and a compact, connected, $\pi_1$-injective subsurface $S\subset\hat{F}$ such that $p_*(\pi_1S)=H$ and $p_*(\pi_1\hat F)\cap B=\emptyset$ and $F\setminus S$ is connected.
\end{theorem}

The crucial ingredient for constructing a prefabricated manifold 
is a supply of QF surface subgroups with the property that for every cusp $V$ of $M$,
there are (at least) two cusps of this collection that are contained in $V$ and have different slopes. This is
used to ensure property (P3). In turn, (P3) implies that $ \bdy \prefab$ 
contains no accidental parabolics; see \refcor{PrefabConvex} and \refprop{PrefabBoundary}.

\section{Ubiquitous closed surfaces}\label{Sec:ClosedUbiquitous}

In this section, we prove \refthm{ClosedQFSurfaces}. The proof is a modification
of
the proof of  \cite[Theorem 4.2]{BCQFS}. In that proof, Baker and Cooper construct a QF surface $\Sigma$  from
a pair of (possibly disconnected) cusped surfaces $\Jcal_1, \Jcal_2$, with all components quasi-Fuchsian. The surface $\Sigma$
is the boundary of a prefabricated manifold $\prefab$, and is obtained by
gluing together subsurfaces
 of the components of $\Jcal_1$ and $\Jcal_2$, together with subsets of horotori. In \cite{BCQFS}
the components of $\Jcal_1$ and $\Jcal_2$ are produced from ideal points of character varieties and group actions on trees, via the work of Culler and Shalen \cite{CS1}. 

In our setting, $\Jcal_1$ and $\Jcal_2$ will be cusped surfaces
produced by \refthm{CuspedSurfacePancake}. 
Given a pancake $P \subset \HH^3$, \refthm{CuspedSurfacePancake} allows us to assume
 that $S = J_1$ has an elevation to $\HH^3$ that strongly separates  $P$.
Thus $S$ has a finite cover containing a very large embedded disk $G$, whose elevation $\widetilde{G} \subset \HH^3$
also separates $P$. By a slight modification of 
 the proof of  \cite[Theorem 4.2]{BCQFS}, we may ensure $\Sigma$
lies extremely close to  $G$, hence $Q(\Sigma)$ also has an elevation 
that separates a slightly thicker pancake ${\Ncal}_{\epsilon}(P)$. 
This ensures the \emph{ubiquitous} condition in \refthm{ClosedQFSurfaces}.

Immersed curves $\alpha, \beta$ in a torus $T$ \emph{have an essential intersection} if they map to multiples of distinct slopes, or  equivalently if the homology classes $[\alpha], [\beta]$ are linearly independent in $H_1(T)$. The definition extends to cusps of surfaces mapping to a rank--$2$ cusp of a $3$--manifold $M$.

\begin{theorem}\label{Thm:ConstructPrefab} 
Suppose $M$ is a cusped hyperbolic $3$--manifold
and $\Scal$ is a finite set
of QF surfaces immersed in $M$. Suppose that 
for each cusp $V$ of $M$, there are at least two cusps of surfaces in $\Scal$ that
both map into $V$ and have an essential intersection in $V$. 

Suppose $P = P(\eta, r)$ is a pancake in $\HH^3$, and
$S$ is a surface in $\Scal$, such that an elevation of $Q(S)$ to $\HH^3$ strongly separates $P$. Then, for any $\epsilon>0$, 
 there is a 
closed QF surface $\Sigma$ immersed in $M$ and an elevation of $\core(\Sigma)$ that strongly separates ${\Ncal}_{\epsilon}(P)$. 
Furthermore, $\Sigma \subset \bdy Z$, where $Z$ is a prefabricated manifold whose quasi-Fuchsian pieces are thickenings of covers of surfaces in $\Scal$.
\end{theorem}

 \begin{proof} 
 For most of the proof, we restrict attention to the following special case.
Suppose that $\Scal$ consists of exactly two immersed QF surfaces, $f_1 \from J_1 \to M$
 and $f_2 \from J_2 \to M$.  Moreover, suppose that  every cusp of $J_1$ has an essential intersection
 with some cusp of $J_2$, and vice versa.
Finally, suppose 
that $S = J_1$.

At the end of the proof, we will briefly describe the (purely notational) changes needed to address the general case. For now, we reassure the reader that the special case described above is all that  is needed in the proof of Theorems~\ref{Thm:ClosedQFSurfaces} and \ref{Thm:CuspedSurfacesOneSlope}.
 
 The maps $f_i \from J_i \to M$ extend to locally isometric immersions $f_i\from Q_i\rightarrow M$, where $Q_i=Q(J_i)$.
There is a decomposition of $M$ into a compact set $K$ and a union of horocusps $\Vcal$,  such that for each component $V\subset \Vcal$
and for each $i\in\{1,2\}$, the preimage $f_{i}^{-1}(V)$ is  a non-empty
union of vertical rank--1 cusps in $Q_i$. Moreover, each component of $J_i\cap f_{i}^{-1}(V)$ is an excellent
annulus. By a small isotopy, we may arrange that $f_1 \vert_{J_1}$ is transverse to $f_2 \vert_{J_2}$. 

Then $\Dcal_i=J_i\cap f_i^{-1}(\bdy\Vcal)$ is a finite set of disjoint, horocyclic simple closed curves. These horocyclic curves cut off
the cusps of $J_i$. Moreover $\Mcal=f_1(\Dcal_1)\cap f_2(\Dcal_2)\subset\bdy\Vcal$ is a finite set and $|f_i^{-1}(x)\cap \Dcal_i|=1$ for each $x\in \Mcal$.   
 The hypotheses imply that
that $f_i^{-1}(\Mcal)$ contains at least one point in every component of $\Dcal_i$. This ensures the \emph{ample spiders} condition formulated in \cite[Definition 3.6, part (W4)]{BCQFS}, which is crucial for the proof of  \cite[Theorem 4.2]{BCQFS}.

From here, we follow the proof of  \cite[Theorem 4.2]{BCQFS}, starting at the second paragraph, with $\Jcal_i$ consisting of the single connected surface $J_i$.
 That proof constructs a prefabricated manifold
 $\prefab={\mathcal C}\cup \Qcal_1\cup \Qcal_2$, where each component of $\Qcal_i$ is a 
finite cover of $Q_i$,  and each component of $\Ccal$ has a thickening that is a finite cover of a component of
     $\Vcal$. That proof ends by verifying the hypotheses of \refcor{PrefabConvex}, which implies that $\prefab$ has a convex thickening $\prefab^+$.
 
 This construction of \cite[Theorem 4.2]{BCQFS} involves a parameter $\delta>0$ with the following meaning. For $i\in\{1,2\}$ the distance in $\prefab$ between distinct components of $\Qcal_i$
 is at least $\delta$. Moreover, the conservative separability \refthm{Conservative} ensures that increasing $\delta$ does not change the number of \emph{pieces} $k + 1 = |\Ccal|+|\Qcal_1|+|\Qcal_2|$ in the construction. (More precisely, this follows from the spider pattern theorem \cite[Theorem 2.8]{BCQFS}.)
 As a result, one eventually achieves $\delta > 8k$, hence \refcor{PrefabConvex} ensures that $Z$ has a convex thickening $\prefab^+$ with $\bdy \prefab^+$ consisting of closed QF surfaces.
 
For our purposes, we modify the construction slightly to ensure the pancake condition in the statement of the theorem.
This requires increasing $\delta$ even further than what is needed for \refcor{PrefabConvex}, but again without increasing the number of pieces in $Z$.
 
Recall that $S = J_1$ has a convex thickening $Q(S)$, with an elevation $\widetilde{Q(S)}$ that strongly separates a pancake $P = P(\eta, r)$. Let $A$ be the rotational axis of $P$, and let $\widetilde{x} \in \bdy \widetilde{Q(S)} \cap A$. Let $x \in \bdy Q(S)$ be the projection of  $\widetilde{x}$.
For the value $\epsilon > 0$ in the statement of the theorem,  fix $R_{\rm big} = R(\epsilon) + r$, where $R(\epsilon)$ is  as in \refthm{CombAsymmetric} and $r$ is the radius of the pancake.
 
Now, we construct a prefabricated manifold $\prefab$ exactly as above, with the same number of pieces. 
Recall that $\Qcal_1$ is a disjoint union
 of finite covers of $Q(S) = Q(J_1)$.
By increasing $\delta$, we ensure that  in some component $\hat Q$ of $\Qcal_1$, the point $x$ has a preimage $\hat x$ whose distance in $\prefab$ from every component of $\Qcal_2$ and $\Qcal_1 \setminus \hat Q$ is larger than $R_{\rm big}$.
 By shrinking the cusps 
 $\Ccal$, we may also ensure the distance (in $\prefab$) from $\hat x$ to every
 component of $\Ccal$ is larger than $R_{\rm big}$. In short, $\hat x$ lies further than $R_{\rm big}$ from every piece of $\prefab$ except its own.
 
 By \refcor{PrefabConvex}, $\prefab$ has a convex thickening $Y$, which may be taken to contain $\Th_8(\prefab)$.
Applying \refthm{CombAsymmetric} to $\prefab$, with $\hat Q$ playing the role of $M_0$, gives
\[
\prefab^+ := \CH(\prefab)  \subset \neb_\epsilon(\hat Q; \, Y) \cup \neb_{R(\epsilon)}\big( (\Qcal_1 \setminus \hat Q) \cup \Qcal_2 \cup \Ccal ; \: Y\big).
\]
Since $\hat x$ has distance greater than $R_{\rm big} = R(\epsilon) + r$ to any other piece of $\prefab$ besides $\hat Q$, it follows that
\[
\neb_r(\hat x ; \, \prefab^+) \subset \neb_\epsilon(\hat Q; Y) \subset \Th_\epsilon(\hat Q).
\]

 Let $\Sigma$ be a component of
 $\partial \prefab^+$ that passes $\epsilon$--close to $\hat x$. Then $\core(\Sigma) \subset \prefab^+ = \CH(\prefab)$, hence
\[
\core(\Sigma) \cap \neb_r(\hat x) \subset \Th_\epsilon(\hat Q).
\]
Now, the projection $Z^+ \to M$ immerses $\core(\Sigma)$ in $M$, sending $\hat x$ to $x$. Choose an elevation of $\core(\Sigma)$ to $\HH^3$ so that $\hat x$ lifts to $\widetilde{x}$. Recall that $\widetilde{x}$ lies on the axis $A$ of $P = P(\eta, r)$. By the above equation, 
\[
\widetilde{\core}(\Sigma) \cap \neb_r(\widetilde x) \subset \Th_\epsilon(\widetilde{ Q(S)}).
\]
Since $\widetilde Q(S)$ strongly separates $P = P(\eta, r)$, it follows that $\widetilde{\core}(\Sigma)$ strongly separates $\neb_\epsilon(P)$.
 
Finally, we discuss how to prove the theorem in the general case. The proof is exactly the same, except that  the  connected surfaces $J_1$ and $J_2$ are replaced by $\Jcal_1$ and $\Jcal_2$, where each $\Jcal_i$ is a (separate) copy of the finite set of immersed surfaces $\Scal$. Then $\Jcal_i$ satisfies the \emph{ample spiders condition}, and can be inserted into 
 the proof of \cite[Theorem 4.2]{BCQFS} exactly as above. In fact, the proof of \cite[Theorem 4.2]{BCQFS} is already adapted to finite collections of surfaces, and contains all the necessary book-keeping notation for keeping track of components of $\Jcal_i$. At the end of the construction of \cite[Theorem 4.2]{BCQFS}, one needs to define $R_{\rm big}$ exactly as above, and argue in the same way that a surface $\Sigma \subset \bdy \prefab^+$ lies very close to $S$ on a disk of big radius.
 \end{proof}

 \begin{proof}[Proof of \refthm{ClosedQFSurfaces}] 
By the Kahn--Markovic theorem \cite{KahnM}, it suffices to treat the case where $M$ has cusps.
Let $\Pi, \Pi' \subset \HH^3$ be a pair of planes whose distance is $4\eta$. Let 
 $P^+ = P(2\eta, r) = \neb_{2\eta}(D_r)$ be a pancake as in \reflem{Pancake}. Let $P = P(\eta, r) = \neb_{\eta}(D_r)$ be a thinner pancake, such that $P^+ = \neb_\eta(P)$.

Let $M = \HH^3 / \Gamma$ be a cusped hyperbolic $3$--manifold. For every cusp $V \subset M$, select a pair of slopes $\alpha(V) \neq \beta(V)$. By \refthm{CuspedSurfacePancake}, there are immersed QF surfaces $J_1 \to M$ and $J_2 \to M$, such that each $J_i$ has an elevation that strongly separates $P$, and each has a cusp mapping to a multiple of every $\alpha(V)$ and every $\beta(V)$.
Since every slope on $\bdy V$ intersects either $\alpha(V)$ or $\beta(V)$ or both, it follows that every cusp of $J_1$ has an essential intersection with some cusp of $J_2$, and vice versa.

 Now, apply
 \refthm{ConstructPrefab} with  $\Scal = \{J_1, J_2 \}$ and with $\epsilon = \eta$.  That theorem produces a closed QF surface $\Sigma$ such that an elevation of $\core(\Sigma)$ strongly separates $P^+ = \neb_\epsilon(P)$. By \reflem{Pancake}, this elevation also separates $\Pi$ from $\Pi'$.
  \end{proof}

\section{Ubiquitous surfaces with prescribed immersed slope}\label{Sec:UbiquitousSlope}

In this section, we prove \refthm{CuspedSurfacesOneSlope}, producing a ubiquitous collection of surfaces with prescribed immersed slope. The argument proceeds in two stages. We begin  by proving \refthm{QFslope}, 
 which says that given a slope $\alpha$ in some cusp
of  $M$, there is a QF surface $F$ immersed in $M$
 and an integer $m>0$ such that all the cusps of $F$ have slope $m \cdot \alpha$. 
 
In the special case where $M$ has one cusp, \refthm{QFslope} follows from a result of Przytycki and Wise \cite[Proposition 4.6]{Przytycki-Wise:mixed-3manifolds}. If $M$ has multiple cusps, the result appears to be new.
%
In fact, \refthm{QFslope} can be proved without appealing to the Kahn--Markovic theorem \cite{KahnM}, by using only the gluing idea behind prefabricated manifolds \cite{BCQFS}. See \refrem{NoKM}.

Next,
we prove  \refthm{CuspedSurfacesOneSlope} by another analogue of the argument of \refthm{ConstructPrefab}. We take the single QF surface $F$ produced by  \refthm{QFslope} and the ubiquitous collection of closed QF surfaces produced by \refthm{ClosedQFSurfaces}. Then, we apply the convex combination theorem to glue together a large regular cover of $F$ with some large cover of a closed surface splitting a desired pancake.

\subsection{One surface with prescribed immersed slope}
The idea of of \refthm{QFslope}  is to build a variant of a prefabricated manifold, whose boundary has cusps.
This entails the following altered definition.

\begin{definition}\label{Def:ModPrefab}
A \emph{modified prefabricated manifold} is a connected, metrically complete, finite-volume, hyperbolic $3$--manifold 
$$\prefab=\Ccal\cup \Qcal_1\cup \Qcal_2 .$$
Each component of $\Qcal_i$ and of $\Ccal$ is a convex hyperbolic $3$--manifold called a \emph{piece}. 
Each component of $\Qcal_i$ is a 
QF $3$--manifold with at least one cusp. 
Each component of $\Ccal$ is  a horocusp. There is a component $Q\subset\Qcal_1$ called the {\em special component}.
One or more rank--1 cusps of $Q$  are called the {\em special cusps}. 
These pieces satisfy the 
following conditions for $i\in\{1,2\}$, and for each component $C$ of $\Ccal$:
\begin{enumerate}[(M1)]
\item  $\Qcal_i\cap\Ccal$ is the disjoint union of all the cusps in $\Qcal_i$ that are not special,
\item  $\Qcal_i\cap \bdy C$  is an annulus with core curve $\alpha_i(C)$,  
\item $\alpha_1(C)$ intersects  $\alpha_2(C)$ once transversely,
\item Each component of $\Qcal_1\cap \Qcal_2$ intersects $\Ccal$.
\end{enumerate}
\end{definition}

Thus each component of $\Ccal$ contains one cusp of $\Qcal_1$ and one
cusp of $\Qcal_2$. In addition to  some number of rank--2 cusps (the components of $\Ccal$),  $\prefab$   also has
some rank--1 cusps, namely the special cusps of $Q$. The next result is a minor modifications of  \cite[Proposition 1.6]{BCQFS}. 

\begin{proposition}\label{Prop:PrefabBoundary} Let $\prefab = \Qcal_1 \cup \Qcal_2 \cup \Ccal$ be a modified prefabricated manifold with a convex thickening. 
Then $\partial \prefab \neq \emptyset$, 
and each component of $\bdy \prefab$ is an incompressible surface. 
Moreover, every loop in $\partial \prefab$ with parabolic holonomy is homotopic into a special cusp.
\end{proposition}

\begin{proof} Let $F$ be a surface with non-empty boundary and $\chi(F) < 0$. Following \cite[Section 7]{BC1}, a  \emph{tubed surface} is a 2--complex formed by gluing each component of $\bdy F$ to an essential simple closed curve in a torus. Distinct  components of $\bdy F$ are glued to distinct tori.

Given $\prefab=\Ccal\cup \Qcal_1\cup \Qcal_2$, as in \refdef{ModPrefab},   let $\Ccal^+=\Ccal\sqcup\Ccal'$ where $\Ccal'$ has one
horocusp for each special cusp of $\Qcal_1$. Let $\prefab^+=\prefab\cup\Ccal'$ be the manifold obtained by obtained by isometrically 
gluing each special cusp
of $Q_1$ into a distinct cusp in $\Ccal'$.  It is clear that $\bdy \prefab\ne\emptyset$. 

Set
$\Qcal_1^+=\Qcal_1\cup\Ccal^+$ and $\Qcal_2^+=\Qcal_2\cup\Ccal$. Then $\prefab^+=\Qcal_1^+\cup\Qcal_2^+$.
Each component of
$\Qcal_i^+$ is a geometrically
finite manifold that retracts to a \emph{tubed surface}. Thus $\prefab^+$ is the union of convex submanifolds, each of which 
retracts to a tubed surface. 
The proof that  $\bdy \prefab^+$ is  incompressible is now the same as that of \cite[Proposition 1.6]{BCQFS}.
Since $\bdy\prefab$ with the cusps truncated is an incompressible subsurface of $\bdy \prefab^+$, it follows that $\bdy\prefab$ is incompressible in $Z$.

Suppose that a loop $\gamma \subset \bdy \prefab$ has parabolic holonomy. Since $\prefab$ has a convex thickening, $\gamma$ must be homotopic into some cusp of $\partial\prefab$. Suppose, for a contradiction, that $\gamma$ is homotopic to a loop $\beta \subset  \bdy C$ for a cusp $C \subset \Ccal$. 
By property (M2), the intersection $\Qcal_i \cap \partial C$ is an annulus, and by  property (M3) the core curves $\alpha_1(C)$ and 
 $\alpha_2(C)$ of these annuli
 have intersection number $1$.
  It follows that
$\beta$ has  intersection number $n \neq 0$ with at least one of $\alpha_1(C)$ and 
 $\alpha_2(C)$. 
Furthermore, $n$ depends only on the homology class $[\beta] = [\gamma] \in H_1(\prefab)$. Since $\gamma$ is disjoint from the mid-surfaces of $\Qcal_1$ and $\Qcal_2$, it follows that $n=0$, which
contradicts the hypothesis that $C \in \Ccal$. 

Since $\gamma \subset \bdy \prefab$ cannot be homotopic into $\Ccal$, it must be homotopic into a special cusp.
\end{proof}

\begin{theorem}\label{Thm:QFslope} Suppose $M=\HH^3/\Gamma$ is a cusped hyperbolic $3$--manifold and  $\alpha$ is a 
slope on a cusp of $M$.  Then there is a cusped quasi-Fuchsian surface immersed in $M$ 
with immersed slope $\alpha$. 
\end{theorem}

As mentioned above, the $1$--cusped case of \refthm{QFslope} is due to Przytycki and Wise  \cite[Proposition 4.6]{Przytycki-Wise:mixed-3manifolds}.
Before giving the full proof of the theorem, we outline the main steps:

 \begin{enumerate}[1.]
 \item Construct a prefabricated manifold $\prefab = \Ccal \cup \Qcal_1 \cup \Qcal_2$, with a convex thickening and a local isometry into $M$. We arrange things so that $\Qcal_1$ contains a piece $Q$ with a cusp mapping to $m \alpha$.
 
 \smallskip
 \item Modify some of the pieces of $\prefab$. We replace $Q \subset \Qcal_1$ by a $3$--fold cyclic cover $\hat Q$. We also replace some of the rank--2 cusps in $\Ccal$ by cyclic $3$--fold covers. 
 The result is a modified prefabricated manifold $\prefab' = \Ccal' \cup \Qcal_1' \cup \Qcal_2$, whose special cusps map to $m \alpha$.

\smallskip
\item Pick a cusped component $F \subset \bdy \prefab'$, and surger it until has the right properties. \refprop{PrefabBoundary} gives a way to remove accidental parabolics by surgery, and \refthm{BonahonCanary} ensures the resulting surface is quasi-Fuchsian.
\end{enumerate}

\begin{proof}[Proof of \refthm{QFslope}]
We claim that  there is a pair of immersed QF surfaces $f_1 \from J_1 \to M$ and $f_2 \from J_2 \to M$, 
such that every cusp of $J_1$ has an essential intersection with some cusp of $J_2$, and vice versa. Furthermore some cusp of $J_1$ is mapped to  $m \cdot  \alpha$ 
for some $m \neq 0$. 
This claim follows immediately from   \refthm{CuspedSurfacePancake}. See also \refrem{NoKM} for an alternate argument.

Now, we plug  the immersed QF manifolds $f_1 \from J_1 \to M$ and $f_2 \from J_2 \to M$ into the construction of
\cite[Theorem 4.2]{BCQFS}. By that theorem, there is a prefabricated manifold $\prefab=\Qcal_1\cup\Qcal_2\cup\Ccal$, with a convex thickening and a local isometry into $M$, where the quasi-Fuchsian pieces of $\Qcal_i$ are thickenings of covers of $J_i$. 
 By \cite[Theorem 2.16]{BCQFS}, we may build $\prefab$ so that every quasi-Fuchsian piece $Q_i \subset \Qcal_i$ has an even number of cusps, with $| \bdy Q_i| \geq 4$. In addition, the spiders meeting $Q_i$ do not disconnect $Q_i$.

Next, we  modify the pieces of
 $\prefab=\Qcal_1\cup\Qcal_2\cup\Ccal$ 
to create a modified prefabricated manifold $\prefab'=\Qcal_1'\cup\Qcal_2\cup\Ccal'$.

By construction, $\prefab$ contains a quasi-Fuchsian piece $Q\subset \Qcal_1$, with a cusp $W\subset Q$ mapping to $m\cdot\alpha$ for some $m \neq 0$. (This manifold $Q$ is a convex thickening of some cover of $J_1$.) 
There is a connected cyclic 3--fold cover $p\from \hat Q\rightarrow Q$ with the following properties.
First, the cusp $W \subset Q$ has a disconnected preimage 
 $p^{-1}(W)= \hat W \sqcup \hat W' \sqcup \hat W''$ consisting of three isometric lifts of $W$.
Every other cusp $V\subset (Q \setminus W)$ has a connected preimage $\hat{V}=p^{-1}(V)$ that is a 
3--fold cyclic cover of $V$. In addition, each component $X\subset Q\cap\Qcal_2$  lifts to $\hat{Q}$. (By \refdef{Prefab}, each component $X\subset Q\cap\Qcal_2$ is an ideal spider.)
The cover $\hat Q$ can be constructed by cutting the mid-surface $S$ of $Q$ along some carefully chosen arcs that avoid the spiders, taking three copies of the cut-up surface, and reassembling. 
Since $|\bdy Q| \geq 4$ and is even,
the existence of  the arcs with the desired properties and the cover $\hat Q \to Q$  is explained in 
the proof of Case 2 of  \cite[Theorem 2.16]{BCQFS}. See  \cite[Page 1215]{BCQFS}.

    For each spider  $X \subset Q$, we choose one lift $\hat X \subset \hat Q$. Since each cusp of $\Qcal_1$ contains exactly one spider leg by \refdef{Prefab},
    it follows that exactly one of $\hat W,\hat W'$ and $\hat W''$ contains a leg of the chosen  lift of some spider. 
    Label the lifts of $W$ so that $\hat W$  is the component that contains a spider leg.

We can now describe the QF pieces of $\prefab'$, as well as how to glue them.
Define $\Qcal_1'=(\Qcal_1\setminus Q)\cup \hat Q$. Every component  $X\subset \Qcal_1\cap\Qcal_2$ is a spider that
corresponds to an isometric spider $\hat X \subset \Qcal_1'$. If $X\subset Q$, the lift $\hat X \subset \hat Q$ was chosen in the previous paragraph. Otherwise, a spider $X\subset(\Qcal_1\setminus Q)$  corresponds to 
 itself (but is also labeled $\hat X$). Now glue $\Qcal_1'$ to $\Qcal_2$ by isometrically identifying each spider $X\subset\Qcal_2$ with the corresponding
 spider $\hat X \subset \Qcal_1'$ to obtain a manifold $Z^*$.

 It remains to construct a union of cusps $\Ccal'$, and to glue these rank--2 cusps onto 
 $Z^*$ to obtain $Z'$. 
Every leg of every spider  $\hat X \subset \Qcal_1' \cap \Qcal_2$ will run into exactly one rank--2 cusp $\hat C \subset \Ccal'$, with distinct legs terminating on distinct cusps. Furthermore, every component $\hat C \subset \Ccal'$ will
correspond to a rank--2 cusp $C \subset \Ccal$, and is a $1$--fold or $3$--fold cover of $C$.
There are three cases, as follows. 
 
First, if $X \subset (\Qcal_1 \setminus Q) \cap \Qcal_2$, then every leg of $X$ lands on a cusp  $C \subset \Ccal$ disjoint from $Q$. For every such cusp $C$ meeting a leg of $X$, we take an isometric copy $\hat C$ and glue it onto the corresponding leg of $\hat X$. 
    Second, if $X \subset Q \cap \Qcal_2$ has a leg on $W$, there is a unique rank--2 cusp  $C_W\in\Ccal$ with  $W \subset C_W$. Construct an isometric lift $\hat C_W$, and embed $\hat W$ into it. By construction, the above-chosen lift $\hat X$ has a leg on $\hat W \subset \hat C_W$.
This specifies a way to glue $\hat C_W$ onto $Z^* = \Qcal_1' \cup \Qcal_2$.

Third, let $\Ccal_Q$ be the union of all other cusps of $\Ccal$ (which meet $Q$ but are disjoint from $W$).
For every $C \subset \Ccal_Q$, \refdef{Prefab} says that $\bdy C$ contains two simple closed curves $\alpha_1(C)$ and $\alpha_2(C)$ that are the cores of $\Qcal_i \cap \bdy C$. These curves form a basis for $\pi_1(C)$. 
Construct a connected 3--fold cyclic cover $\hat C \rightarrow C$ corresponding to the subgroup 
$\langle 3\alpha_1(C), \, \alpha_2(C) \rangle \subset \pi_1(C)$. Then we may glue $\hat C$ onto $\hat Q \subset \Qcal_1'$ along a neighborhood of $3 \alpha_1(C)$ and onto $\Qcal_2$ along a neighborhood of $\alpha_2(C)$.

We have now constructed a modified prefabricated manifold $\prefab' = \Ccal' \cup \Qcal_1' \cup \Qcal_2$. 
 The special component is $\hat Q$ and
the special  cusps are $\hat W'$ and $\hat W''$.

Since $\prefab$ has a convex thickening, then so does $\prefab'$. In fact, one choice of convex thickening of $\prefab'$ is obtained
by doing the corresponding modifications to the thickened pieces of $\prefab$.  We rename $\prefab'$ to be this convex thickening.
The local isometries $\Qcal_i \to M$ and $\Ccal \to M$ define local isometries $\Qcal_1' \to M$ and $\Qcal_2 \to M$ and $\Ccal' \to M$, which agree on the overlaps. Thus $\prefab'$ has a local isometry into $M$. Note that the special cusps $\hat W'$ and $\hat W''$ both map to $m \alpha$.

Let $F$ be a component of $\partial \prefab'$ that contains a (rank--1) cusp. By \refprop{PrefabBoundary}, 
$F$ is incompressible and every parabolic in $F$ is homotopic to
a special slope.  Any accidental parabolics in $F$ can be removed by surgery, as follows. 
If there is an accidental parabolic in $F$, then Jaco's theorem \cite[Theorem VIII.13]{Jaco}
provides an annulus $A$ embedded in $\prefab$, with one end an essential simple non-peripheral
loop $\gamma \subset F$ and the other end a simple closed curve that is a special slope
in a special cusp. Surgering $F$ using $A$
produces a new surface $F'$ with $\chi(F') < 0$, isotopic to a subsurface $F$, and with fewer accidental parabolics. 
After a finite number of steps one obtains from $F$
 a surface $E$ without accidental parabolics 
that is (isotopic to) a subsurface of $F$ and has at least one cusp. 
By \refprop{PrefabBoundary}, the cusps of $E$ all project to the slope $m\cdot\alpha$ in $M$.

 The modified prefabricated manifold $Z'$ is convex, has finite volume, and has $\bdy Z' \neq \emptyset$. Since $E \subset Z'$ and $\rho \from \pi_1 E \to \pi_1 Z' \subset \pi_1 M$ is type-preserving,  
\refthm{BonahonCanary} implies it is  quasi-Fuchsian.
\end{proof}

\begin{remark}\label{Rem:NoKM}
\refthm{QFslope} can be proved without relying on the Kahn--Markovic theorem. In the above argument, the only appeal to Kahn--Markovic was in the claim at the start:
 that
 there is a pair of immersed QF surfaces $f_1 \from J_1 \to M$ and $f_2 \from J_2 \to M$, 
such that every cusp of $J_1$ has an essential intersection with some cusp of $J_2$, and vice versa. Furthermore, we claimed some cusp of $J_1$ is mapped to  $m \cdot  \alpha$ 
for some $m \neq 0$. 
Here is a sketch of an alternate proof of the claim.

By  \cite[Theorem 4.2]{BCQFS}, there is a prefrabricated 3-manifold $Z=\Qcal_1\cup\Qcal_2\cup\Ccal$
 and a local isometry $g \from Z \to M$. Let $\Fcal_i$ be the midsurface of each $\Qcal_i$, and let $\Fcal=\Fcal_1\sqcup\Fcal_2$.
For each component $F$ of $\Fcal$, let $d(F)$ be a non-zero integer, and let
 $\nu(F)\in\{\pm 1\}$ be an orientation of $F$.
There is an incompressible, oriented, possibly disconnected, surface 
$\Fcal(d)$ properly embedded in $Z$ obtained by cutting and cross-joining $|d(F)|$ copies of each component
$F$ of $\Fcal$ using the orientation $\nu(F)\cdot \mathop{sgn}(d(F))$, and then compressing as much as possible. Finally, one
surgers away accidental parabolics. Then $\Fcal(d)$ is QF by Theorem \ref{Thm:BonahonCanary}, and $[\Fcal(d)]=\sum d(F)[F,\nu(F)]\in H_2(Z)$.
If $C$ is a cusp of $Z$, let $\alpha_i(C)=[\Fcal_i\cap \bdy C]\in H_1(C)$ be the slope determined by the oriented surface $\Fcal_i$
in $C$.  Let $F_i\subset\Fcal_i$ be the component that meets $C$. Then the slope of $\Fcal(d)$ in $C$ is $d(F_1)\alpha_1(C)+d(F_2)\alpha_2(C)$. Thus, by choosing $d$ appropriately, we may arrange that $g(\Fcal(d)) \subset M$ has a prescribed slope in a 
prescribed cusp of  $M$, and meets every cusp. By performing this construction twice, we obtain $J_1$ and $J_2$ as required for the claim.
\end{remark}

\subsection{Intersecting and gluing ubiquitous surfaces}
In the remainder of this section, we employ another cut-and-paste construction to derive \refthm{CuspedSurfacesOneSlope}  from \refthm{QFslope}. Before proceeding, we note that only \refthm{ClosedQFSurfaces} and \refthm{QFslope}, both of which have already been established, are needed to derive the cubulation statement of \refcor{Cubulation} in the next section. 

The proof of \refthm{CuspedSurfacesOneSlope} requires a topological lemma that ensures our surface will not have accidental parabolics.

\begin{lemma}\label{Lem:BookNoAccidental} 
Suppose $F$ and $G$ are compact, connected, orientable surfaces with negative Euler characteristic and non-empty boundary.
In addition, suppose $Y =P\cup Q$ is a compact 3--manifold, where $P= F \times I$ and $Q= G \times I$. 
Suppose $P$ and $Q$ intersect along a collection of annuli $\Acal=P\cap Q= \bdy G \times I \subset  \interior(F) \times \bdy I$. 
Let $\Sigma$ be a component of $\bdy Y \setminus \interior(\bdy F \times  I)$.
Suppose every component $X\subset \Sigma\setminus \Acal$ has $\chi(X)<0$. 

Then $\Sigma$ is incompressible in $Y$. 
If $\gamma \subset \Sigma$ is a loop 
  that is freely homotopic 
into $\bdy\Sigma$ through $Y$, then $\gamma$ is freely homotopic  into $\bdy\Sigma$ through $\Sigma$.
\end{lemma}

\begin{proof}
This is a standard topological argument in the same spirit as \refclaim{NoAccidental}. One needs to consider the intersections between $\Acal$ and an annulus $B \subset Y$ realizing a homotopy of $\gamma$ into $\bdy \Sigma$. The hypothesis $\chi(X) < 0$ for every component $X \subset \Sigma \setminus \Acal$ ensures that intersections $B \cap \Acal$ must indeed occur.
\end{proof}

We also need a geometric lemma about convex hulls. If $Q$ and $N$ are convex subsets in $\RR^3$, then $Q\cap N$ is convex, but $\partial Q$ can undulate in and out of $N$,
so that $N\cap\partial Q$ might be an arbitrary planar surface. For example, consider how a concentric ball and cube might intersect. 
The same phenomenon happens if $Q$ and $N$ are convex submanifolds of a hyperbolic
3--manifold. In particular,  $N\cap\bdy Q$ might be compressible in $\bdy Q$. The situation is better when $Q$ is equal to its convex core. We also need the mild assumption that  $\pi_1 Q$ is finitely generated.

\begin{lemma}\label{Lem:ConvexIntersect} 
Suppose that $Q$ and $N$ are convex 3--dimensional submanifolds of  a complete hyperbolic 3--manifold $Y$, where  $Q \cap N$ is compact,  $Q = \core(Q)$, and $\pi_1 Q$ is finitely generated.
If $X$ is a component of $\bdy Q \setminus \interior(N)$, then $\chi(X)<0$.
\end{lemma}

\begin{proof}  Let $\Gamma$ be the holonomy of $Q$.
The hypothesis $Q = \core(Q) =\CH(\Lambda(\Gamma))/\Gamma$ implies
 that the universal cover $\widetilde Q$ is the convex hull of $\Lambda(\Gamma)$.
Since $Q$ is 3--dimensional, it follows that
 that $\Gamma$ is non-elementary. Since $\Gamma$ is finitely generated, the Ahlfors finiteness theorem \cite{Ahlfors:Finiteness, Bers:AreaBound} says that  every component of $\bdy Q = \bdy \! \core(Q)$ is intrinsically a finite-area hyperbolic surface. (Observe that if $Q$ is quasi-Fuchsian, this conclusion is immediate without citing any theorems.) 
 Therefore, $X \subset \bdy Q$ cannot be homeomorphic to $S^2$ or $T^2$, since such a closed surface would form an entire component of $\bdy Q$. Furthermore, $X$ has finite area.

Now, suppose that $\chi(X) \geq 0$. Then 
 $\bdy X = X\cap N \neq \emptyset$, and $\bdy X \subset N$ consists of one or two circle components.
If $X$ is compact, then it is a disk or annulus. Otherwise,  $X\cong S^1\times[0,\infty)$
is a finite-area end of $\bdy Q$.
We claim that
 $X$ is contained in the convex hull of $\bdy X$. 
 Since $N$ is convex and $\bdy X\subset N$, it follows
that $X\subset N$, which is a contradiction.

  By Carath\'eodory's theorem \cite[Proposition 5.2.3]{Papadopoulos},  
  every point of $\widetilde Q$ is contained
in an ideal  simplex with vertices in 
$\Lambda(\Gamma)$. It follows that each $x\in\bdy \widetilde Q$
 is contained in an ideal  $1$-- or $2$--simplex $\Delta\subset \bdy \widetilde Q$.
Let $\rho \from \pi_1 X \to \pi_1 Q$ be the  inclusion--induced homomorphism.
Since $\pi_1X$ is trivial or cyclic, there are three possibilities.

First, suppose that $\rho(\pi_1 X)$ is generated by a parabolic element. Then some component $\gamma$ of 
$\bdy X$ is homotopic to a horocycle in $Q$. Thus the convex hull $C$
 of $\gamma$ is not compact.
But $\gamma\subset Q\cap N$ and $Q$ and $N$
are both convex, so $C\subset Q\cap N$ is not compact, 
contradicting the hypothesis that $Q\cap N$ is compact.

Next, suppose that $\rho (\pi_1 X)$  is generated by a hyperbolic element. Then the
core curve of $X$ is homotopic to a closed geodesic $\gamma  \subset Q$, and $X$ is a compact annulus. Hence $X$ has an elevation $\widetilde X \subset \widetilde Q$, which lies within a bounded neighborhood
of a geodesic $\widetilde \gamma  \subset \HH^3$ covering $\gamma$. Given $x\in\widetilde X$,
let $\Delta\subset\bdy \widetilde Q$ be an ideal simplex that contains $x$. Each geodesic ray from 
$x$ to a vertex of $\Delta$ either crosses
$\bdy \widetilde X$, or else limits on an endpoint $z$ of $\widetilde\gamma$. Since $z$ is in the limit set of $\bdy\widetilde X$, 
it follows that $X$ is contained in the convex hull of $\bdy X$. 

Next, suppose that $\rho(\pi_1 X)$ is trivial. Then $X$ has an isometric lift $ \widetilde X\subset \bdy\widetilde Q$. 
If $X$ is not compact, then $X\cong S^1\times[0,\infty)$  is 
a finite-area end of $\bdy Q$. In this case, a compressing disk for $X$ lifts to a compressing disk of $\widetilde X$, which cuts off a thorn end of $\widetilde Q$. But then the limit point of this thorn must be an isolated point of $\Lambda(\Gamma)$, 
 which contradicts the well-known fact that $\Lambda(\Gamma)$ is a perfect set whenever $\Gamma$ is non-elementary \cite{Ahlfors:Finiteness}.

The remaining case is that $\rho(\pi_1 X)$ is trivial and $X$ is compact. Then 
each point $x\in \widetilde X$ lies in some ideal simplex $\Delta\subset\bdy \widetilde Q$. 
Since $ \widetilde X = X$ is compact, 
each geodesic ray  from $x$ to
an ideal vertex of $\Delta$ must cross a component of  $\bdy \widetilde X$, so $x$ is in the convex hull of these crossing points in $\bdy \widetilde X$.
This proves the claim and the lemma.
\end{proof}

\begin{remark}\label{Rem:TotallyGeodesicIntersect}
A version of \reflem{ConvexIntersect} also holds if $Q \subset Y$ is a totally geodesic finite-area surface (hence $Q = \core(Q)$) and $N \subset Y$ is a convex submanifold of any dimension such that $Q \cap N$ is compact. Under these hypotheses, any component $X$ of $Q \setminus \interior(N)$ has $\chi(X) < 0$. The proof is easier than \reflem{ConvexIntersect}, because a totally geodesic annulus always lies in the convex hull of its boundary.
\end{remark}

  A  map $f \from A\to B$  is {\em conjugacy injective} if the induced homomorphism $f_*$ sends distinct conjugacy classes in $\pi_1 A$ to distinct conjugacy classes in $\pi_1 B$.
       The inclusion of a subsurface $A$ into a surface $B$ is conjugacy injective if and only if $\chi(X)<0$ for
     any non-peripheral component $X\subset B \setminus A$.

A pair of 
QF manifolds $Q_1$ and $Q_2$ that are isometrically immersed in a complete hyperbolic 3-manifold $M$  \emph{have an essential intersection} 
if there are elevations  $\widetilde Q_1, \widetilde  Q_2 \subset\HH^3$ such that the limit set $\Lambda( \widetilde  Q_1)$ intersects
both components of $S^2_{\infty}\setminus\Lambda( \widetilde Q_2 )$. It is easy to see this condition is symmetric in $Q_1$ and $Q_2$,
and is equivalent to the statement that both components of
$\HH^3\setminus\widetilde Q_2$ contains points in $\widetilde Q_1$ at unbounded distance from $\widetilde Q_2$.

 If $Q$ is a quasi-Fuchsian $3$--manifold, then a finite covering $\hat Q\to Q$
  is \emph{conservative} if $Q$ and $\hat Q$ have the same number of cusps.
 The next result  says one can glue together
 large conservative covers of two QF manifolds to obtain a geometrically finite
 manifold with connected QF boundary containing large parts of covers of the two surfaces.

\begin{proposition}\label{Prop:Glue2QF} 
For $i\in\{1,2\}$  suppose $f_i \from Q_i \to M$ is an isometric immersion of a QF manifold $Q_i$ into a complete hyperbolic $3$--manifold $M$.
 Suppose that $Q_1$ and $Q_2$  have an essential intersection in $M$, and that one of the $Q_i$ is compact.

Fix $K >0$.  Then there exist conservative covers  $p_i \from \hat Q_i\to Q_i$ and 
 a hyperbolic 3-manifold  $\hat Y= \hat Q_1\cup \hat Q_2$, with an isometric immersion $g \from \hat Y\to M$ such that $g \vert_{\hat Q_i} =f_i \circ p_i$.
Moreover, $\hat Y$ has a convex thickening such that $\Ncal_K(\hat Q_1\cap \hat Q_2 ; \, \Th_{\infty}(\hat Y))=\Th_K (\hat Q_1\cap \hat Q_2)$. 
Finally, 
$\bdy \hat Y$ has at most two components, each of which is quasi-Fuchsian and meets every
cusp of $\hat Y$. 
\end{proposition}

\begin{proof} 
  For most of the proof, we will work under the additional hypothesis that each 
 $Q_i$ is not Fuchsian.  In this case, by Lemma \ref{Lem:ConvexCollar}, no additional generality is lost by assuming  $Q_i =\Core(Q_i)$. The Fuchsian case is explained at the end.
 
Fix elevations $\widetilde Q_i \subset \HH^3$ whose limit sets intersect as required, and define $\widetilde N=\widetilde Q_1\cap\widetilde Q_2$.
The virtual convex combination theorem \cite[Theorem 5.4]{BC1} implies that,
 after replacing each $Q_i$ by a certain finite cover,
there is  a connected, geometrically finite hyperbolic 3--manifold $Y = Q_1\cup Q_2$ 
and a local isometric immersion $g\from Y\rightarrow M$ such that $g \vert_{ Q_i} =f_i$, and moreover
$Y$ has a convex thickening. Since one of the $Q_i$ is compact, it follows that $N=Q_1\cap Q_2$ is compact and convex.
The essential intersection hypothesis implies that $Q_1$ meets both boundary components of $Q_2$, 
and vice versa.  By \refthm{Conservative},
we may assume these covers are conservative.
 
 The subgroup $\pi_1N$ is finitely generated (because $N$ is compact),  free (because it is an infinite index subgroup of a surface group), 
 and contains no parabolics (because $N$ is compact and convex). Since $N$ is irreducible, it must be a handlebody. 
 If $A$ is a path connected
 subspace of $B$, and
the map $\pi_1A\to\pi_1B$ induced by inclusion is injective, then we will regard $\pi_1A$ as a subgroup of $\pi_1B$.

Next, we pass to further conservative covers $\hat Q_i \to Q_i$ that satisfy several conditions.

\begin{claim}\label{Claim:ConservProps}
Given $K > 8$, there is a hyperbolic 3-manifold $\hat Y = \hat Q_1 \cup \hat Q_2$ such that 
 \begin{enumerate}[{\rm (Y1)}]
\item $\hat Q_i\to Q_i$ is a conservative cover,
\item $\hat N=\hat Q_1 \cap \hat Q_2$ is an isometric lift of $N$, 
\item the inclusion $\hat N \hookrightarrow \hat Q_i$ extends to an isometric embedding $\Th_K (\hat N) \hookrightarrow \Th_{\infty}(\hat Q_i)$,
\item   $\pi_1\hat N$ is conjugate in $\pi_1\hat Q_i$ to $\pi_1F_{i,\pm}$ for some connected incompressible subsurface $F_{i,\pm}\subset \bdy_{\pm}\hat Q_i$,
\item $F_{i,\pm}$  is compact,   and $\bdy_{\pm}\hat Q_i\setminus F_{i,\pm}$ is connected,
\item $\hat Y$ has a convex thickening $\Th_{\infty}(\hat Y)$.
\end{enumerate}
\end{claim}

Indeed, conclusions (Y1)--(Y5) follow by applying \refthm{Conservative}. 
Meanwhile, once we set $K > 8$, condition (Y6) follows from the convex combination theorem \cite[Theorem 1.3]{BCQFS}.

Next, we show that each component of $\bdy \hat Y$ is QF.
To do this, we will replace $ \hat Y$
by a certain thickening.
Define $\hat Q_2^+= \Th_K(\hat Q_2)$  and  $\hat N^+= \hat Q_1\cap \hat Q_2^+$ and  $\hat Y^+= \hat Q_1\cup \hat Q_2^+$.
Since $\hat Y$ and $\hat Y^+$ have a common thickening, it follows from \reflem{ConvexCollar} that they have isotopic boundaries.

By (Y5), we know that  $\bdy_{\pm} \hat Q_1\setminus F_{1,\pm}$ is connected. Set $F'_{1, \pm}= \bdy_{\pm}\hat Q_1\cap \hat Q_2^+$. As $K$ increases, the set $F'_{1, \pm}$ also increases. Thus
  we may assume that $K$ is large enough to ensure that $F_{1,\pm} \subset F'_{1,\pm}$.
As $K$ increases, distinct components of $F'_{1,\pm}$ may combine, but because 
  $\hat Q_1$ and $\hat Q_2^+$ are
  convex, new components cannot be
  created. Thus we may assume $K$ is also chosen so that $F'_{1,\pm}$
    is connected.
   Since $\hat Q_1= \core(\hat Q_1)$,  it follows from
 \reflem{ConvexIntersect} that no component of $\bdy_{\pm} \hat Q_1\setminus F_{1,\pm}'$ is a disk or annulus.
  Thus $F'_{1,\pm}$ is just $F_{1,\pm}$ with  collars added onto the boundary components. We now replace $F_{1,\pm}$ by $F_{1,\pm}'$. Since $F_{1,\pm}$ is conjugacy injective in $\bdy_{\pm} \hat Q_1$, it follows that $\hat N$ is conjugacy injective in $\hat Q_1$. The same argument, interchanging the roles of $Q_1$ and $Q_2$, shows that  $\hat N$ is
conjugacy injective in $\hat Q_2$. 

For notational simplicity, we now replace $\hat Q_2$ by $\hat Q_2^+$ and $\hat N$ by $\hat N^+$ and $\hat Y$ by $\hat Y^+$.  Since $\hat N$ is a handlebody, it is irreducible, and
  it follows from the $h$--cobordism theorem for 3--manifolds that there is a homeomorphism $h \from F\times [-1,1]\to \hat N$ with $h(F\times(\pm 1))= F_{1,\pm}$.
  The \emph{vertical boundary of $\hat N$} is $\bdy_v \hat N=h(\bdy F\times [-1,1])$. Observe that $\bdy_v \hat N \subset\bdy \hat Q_1 \cup \bdy \hat Q_2$.
  
  If $\bdy_v \hat N$ is not entirely contained in $\bdy \hat Q_2$, then
  $\bdy_v \hat N \cap \bdy \hat Q_1\neq \emptyset$ and we can isotop $\bdy \hat Q_2$, shrinking $\hat Q_2$, and keeping $\hat Q_2$ convex, 
  until $\bdy_v \hat N\subset \bdy \hat Q_2$.
 Let $X$ be a  component of $\bdy \hat Q_2\setminus\interior(\bdy_v \hat N)$. 
 Since $\hat N$ is conjugacy
 injective in $\hat Q_2$, it follows that either $\chi(X) < 0$ or $X$ is peripheral in $\bdy \hat Q_2$. But $X$ cannot be peripheral because $\bdy X \subset \hat N$ cannot be parabolic. Thus $\chi(X)<0$.

  Define $A_{\pm}=\cl(\bdy_{\pm}\hat Q_1\setminus F_{1,\pm})$, and set $P := \cl(\hat Q_1\setminus \hat N) \cong A_+\times I$.
 We wish to apply \reflem{BookNoAccidental} 
  with $Q = \hat Q_2$ and $P$ as above to deduce that  $\bdy \hat Y^+$ (and hence $\bdy\hat Y$) is incompressible and contains no accidental parabolics. However, one of 
  $P$ or $Q$ might not be compact,
 in which case the ends are rank--$1$ cusps. Since $N$ is compact, we may truncate any cusps of $P$ or $Q$
 away from $N$ and replace the manifold by a homotopy equivalent compact one before appealing to \reflem{BookNoAccidental}.
Now, it follows from \refthm{BonahonCanary} that every component of $\bdy Y$ is QF.
  
Observe that $\bdy \hat Y^+=A_+\cup A_-\cup B$, where  $B\subset\bdy \hat Q_2$. 
Each of $A_{\pm}$ is connected by (Y5), and
each component of $B$ intersects at least one component of $A_{\pm}$. Hence $\bdy \hat Y^+$ (and so $\bdy\hat Y$) 
has at most two components. Finally, if $Y$ contains cusps, we may relabel the $Q_i$ so that $Q_1$ contains cusps. 
Since $F_{1,\pm}$ is compact, each of $A_{\pm}$ meets every cusp of $\hat Y$, hence every component of  $\bdy\hat Y$ meets every cusp of $\hat Y$.

If some $Q_i$ is Fuchsian, we merely need to modify the argument that establishes the conjugacy injectivity of $\hat N$ in $\hat Q_i$. The argument stays the same up to the construction of $\hat Q_i$ in \refclaim{ConservProps}. If  $\hat J_1 = \core(\hat Q_1)$ is Fuchsian, then $\hat N = \hat Q_1 \cap \hat Q_2$ deformation retracts onto the compact Fuchsian surface
 $\hat J_1 \cap \hat Q_2 = \hat J_1 \cap \hat N$. By \refrem{TotallyGeodesicIntersect}, every component of $\hat J_1 \setminus \hat N$ has negative Euler characteristic. Thus $\hat N \cap \hat J_1$ is conjugacy injective in $\hat J_1$, implying that $\hat N$ is conjugacy injective in $\hat Q_1$. The rest of the proof is the same.
\end{proof}

   We can now complete the proof of \refthm{CuspedSurfacesOneSlope}.

\begin{proof}[Proof of \refthm{CuspedSurfacesOneSlope}] 
Let $M = \HH^3 / \Gamma$ be a cusped byperbolic $3$--manifold, and let  $\alpha$ be a slope on one cusp of $M$.  Let $\Pi_-$ and $\Pi_+$ be a pair of planes in $\HH^3$ whose distance is $4\eta >0$. Let $P^+ = P(2 \eta, r) = \neb_{2 \eta}(D_r)$ be a pancake as in \reflem{Pancake}.    Set $\epsilon = \eta$ and let $P = P(\eta, r)$ be a thinner pancake, so that $P^+ = \neb_\epsilon(P)$. Let $R = R(\epsilon)$ be the constant of \refthm{CombAsymmetric}.

By \refthm{QFslope}, there is an immersed QF surface $f_1 \from J_1 \rightarrow M$
with immersed slope $\alpha$. Let $Q_1 = Q(J_1)$.
By the work of Shah \cite{Shah:Closures} and Ratner \cite{Ratner:Topological}, there is a hyperbolic plane $\Pi' \subset \HH^3$ that strongly separates $P$,  and has an essential intersection with some elevation $\widetilde  Q_1 $. We also require that $\widetilde  Q_1$ lies very far from the pancake $P$; specifically, $d(P, \widetilde Q_1) > R$.
Observe that
 any small perturbation of $\Pi'$ also separates $P$ and has an essential intersection with the same elevation $\widetilde  Q_1$.

By \refthm{ClosedQFSurfaces}, closed quasi-Fuchsian surfaces are ubiquitous in $M$. Thus there is an
 immersed closed QF surface $f_2 \from J_2 \to M$ with an elevation that lies arbitrarily close to $\Pi'$. In particular, $Q_2 =Q(J_2)$
  has an elevation $\widetilde Q_2$ that strongly separates the pancake $P$ and has limit points in both components of $S^2_\infty \setminus \widetilde  Q_1 $.
  
   Apply \refprop{Glue2QF} with $Q_1 $ and $Q_2$ as above, and with $K = d(P, \widetilde Q_1) > R$, to obtain a  hyperbolic 3--manifold $\hat Y = \hat Q_1 \cup \hat Q_2$ whose convex thickening $\hat Y^+$ has QF boundary. By \refthm{CombAsymmetric}, the portion of $\hat Y^+ = \CH(\hat Y)$ that   is $R$--far from $\hat Q_1$ must lie $\epsilon$--close to $\hat Q_2$. Thus there is a component $\Sigma \subset \bdy \hat Y^+$ with an elevation $\widetilde \Sigma$ that lies $\epsilon$--close to $\bdy \widetilde Q_2$ on a region that includes the pancake $P$. Therefore, 
 $\widetilde \core( \Sigma)$ strongly separates $P^+ = \neb_\epsilon(P)$.

By  \refprop{Glue2QF}, the component $\Sigma$ has cusps, 
all of which are cusps of  $\bdy \hat Q_1$. Since
the cover $\hat Q_1 \to Q_1$ used to construct $\hat Y$ is conservative,  all cusps of $\Sigma$ map to the same multiple of $\alpha$.
   \end{proof}

\section{Cubulating the fundamental group}\label{Sec:Cubulation}

In this section, we explain how \refcor{Cubulation} follows from the preceding theorems and \cite{BergeronWise, Hruska-Wise:RelativeCocompact}. We begin by briefly reviewing  the terminology associated to cube complexes and the groups that act on them. The references \cite{Hruska-Wise:RelativeCocompact, Sageev:survey, Wise:Riches2Raags} give an excellent and detailed description of this material.

For $0 \leq n < \infty$, an \emph{$n$--cube} is $[-1, 1]^n$. A \emph{cube complex} is the union of a number of cubes, possibly of different dimensions, glued by isometry along their faces. A cube complex $X$ is called  \emph{CAT(0)} if it is simply connected, and if the link of every vertex is a flag simplicial complex. By a theorem of Gromov, this combinatorial definition is equivalent to the CAT(0) inequality for geodesic triangles.

A \emph{midcube} of an $n$--cube is the $(n-1)$ cube obtained by restricting one coordinate to $0$. A \emph{hyperplane} $H \subset \widetilde X$ is a connected union of midcubes, with the property that $H$ intersects every cube of $\widetilde X$ in a midcube or in the empty set. Hyperplanes in a CAT(0) cube complex $\widetilde X$ are embedded and two-sided, hence they can be used to inductively cut $\widetilde X$ (and its quotients) into smaller pieces. This endows cube complexes and the groups that act on them with a hierarchical structure. See Wise \cite{Wise:Riches2Raags}, where this philosophy is extensively fleshed out.

Suppose that a group $G$ acts freely (that is, without fixed points) on a CAT(0) cube complex $\widetilde X$. 
Then the quotient $X = \widetilde X / G$ is a \emph{non-positively curved} cube complex.
The quotients of hyperplanes in $\widetilde X$ are \emph{immersed hyperplanes} in $X$.
If the immersed hyperplanes of $X$ are embedded and two-sided and avoid two other pathologies (see \cite[Definition 4.2]{Wise:Riches2Raags}), $X$ and $G$ are called \emph{special}. By a theorem of Haglund and Wise, $G$ is special if and only if it embeds into a right-angled Artin group \cite{Haglund-Wise:Special}.

Following Sageev \cite{Sageev:EndsOfGroups}, group actions on cube complexes can be constructed in the following way. Suppose that $G = \pi_1 Y$, where $Y$ is a compact cell complex. 
Then $G$ acts by deck transformations on $\widetilde Y$, and is quasi-isometric to $\widetilde Y$. 
Suppose that $H_1, \ldots, H_k$ are \emph{codimension--$1$} subgroups of $G$, meaning that some metric thickening of an orbit $(H_i) y$ separates $\widetilde Y$ into non-compact components. 
Sageev uses this data to build a $G$--action on a \emph{dual cube complex} $\widetilde X$, whose hyperplanes are in bijective correspondence with cosets of the $H_i$. See Hruska and Wise \cite{Hruska-Wise:RelativeCocompact} or Sageev \cite{Sageev:survey} for detailed, self-contained expositions.

In our application, $Y$  is the compact part of a hyperbolic $3$--manifold $M$. The codimension--$1$ subgroups of $G = \pi_1 Y = \pi_1 M$  are QF surface groups. The quasi-isometry $G \to \widetilde Y$ identifies the cosets of $\pi_1 S $ with the elevations of $S$ in $\widetilde Y$. These elevations give rise to hyperplanes in $\widetilde X$.


The following is a special case of a theorem of Bergeron and Wise \cite[Theorem 5.1]{BergeronWise}.

\begin{theorem}\label{Thm:Proper}
Let $M = \HH^3/\Gamma$ be a cusped hyperbolic $3$--manifold. Suppose that $M$ contains an ubiquitous collection $\Scal$ of quasi-Fuchsian surfaces, with the property that for every cusp $V \subset M$, the cusps of surfaces in $\Scal$ map to at least two distinct immersed slopes in $V$. Then there are finitely many surfaces $S_1, \ldots, S_k \in \Scal$ such that $\Gamma = \pi_1 M$ acts freely on the CAT(0) cube complex $\widetilde X$ dual to all elevations of $S_1, \ldots, S_k$ to $\HH^3$.
\end{theorem}

Combining \refthm{Proper} with \refthm{CuspedSurfacesManySlopes} gives

\begin{corollary}\label{Cor:WeakCubulation}
Let $M = \HH^3 /\Gamma $ be a cusped hyperbolic $3$--manifold. Let $\Scal$ be the collection of surfaces guaranteed by \refthm{CuspedSurfacesManySlopes}. 
Then $\Gamma = \pi_1 M$ acts freely on a finite dimensional CAT(0) cube complex $\widetilde X$, with finitely many $\Gamma$--orbits of hyperplanes. Each immersed hyperplane in $X = \widetilde X/ \Gamma$ corresponds to an  immersed, cusped quasi-Fuchsian surface in $\Scal$.
\end{corollary}

We emphasize that in \refcor{WeakCubulation},  the quotient $X = \widetilde X / \Gamma$ need not be compact. However, this quotient has finitely many immersed hyperplanes. Such a cubulation of $\Gamma$ is called \emph{co-sparse}. While weaker than the cocompact cubulation described below, a co-sparse cubulation can frequently be promoted to be virtually special. See \cite[Proposition 3.3]{BergeronWise}  and \cite[Theorem 15.10]{Wise:Hierarchy}. 

Stronger hypotheses on $\Scal$ guarantee that the action on $\widetilde X$ is cocompact. The following result is implicit in Hruska and Wise  \cite[Theorem 1.1]{Hruska-Wise:RelativeCocompact}.

\begin{theorem}\label{Thm:ProperCocompact}
Let $M = \HH^3/\Gamma$ be a cusped hyperbolic $3$--manifold. Suppose that $M$ contains an ubiquitous collection $\Scal$ of quasi-Fuchsian surfaces, with the property that for every cusp $V \subset M$, the cusps of surfaces in $\Scal$ map to exactly two distinct  immersed slopes $\alpha(V)$ and $\beta(V)$. Then there are finitely many surfaces $S_1, \ldots, S_k \in \Scal$ such that $\Gamma = \pi_1 M$ acts freely and cocompactly on the CAT(0) cube complex $\widetilde X$ dual to all elevations of $S_1, \ldots, S_k$ to $\HH^3$.
\end{theorem}

\begin{proof}
Let $S_1, \ldots, S_k$ be the finite collection of surfaces guaranteed by \refthm{Proper}. Then $\Gamma$ acts freely on the cube complex $\widetilde X$ dual to these surfaces. 
Let $V_1, \ldots, V_n$ be the horocusps of $M$.

Hruska and Wise  \cite[Theorem 1.1]{Hruska-Wise:RelativeCocompact} show that the $\Gamma$--action on the dual cube complex $\widetilde X$ is \emph{relatively cocompact}. 
This means that the quotient $X = \widetilde X/\Gamma$ decomposes into a compact cube complex $K$ and cube complexes $C_1, \ldots, C_n$ corresponding to the cusps $V_1, \ldots, V_n$. Each $C_j$ is the quotient under $\pi_1(V_j) \cong \ZZ^2$ of a cube complex $\widetilde{C_j} \subset \widetilde X$. Furthermore, $\widetilde{C_j} \subset \widetilde X$ is exactly the dual cube complex constructed from the parabolic subgroups $\pi_1(S_i) \cap \pi_1(V_j)$ acting on $\widetilde{V_j}$.

In the case at hand, every peripheral group $\pi_1(V_j) \cong \ZZ^2$ can intersect a conjugate of  $\pi_1(S_i)$ in one of two parabolic subgroups, namely $\langle \alpha(V_j) \rangle$ or $\langle \beta(V_j) \rangle$. The slopes $\alpha(V_j)$ and $\beta(V_j)$ are distinct. Thus $\widetilde{C_j}$ is quasi-isometric to the standard cubulation of $\RR^2$, constructed from two families of parallel lines in the plane. See \cite[Figure 6.4]{Wise:Riches2Raags}. It follows that every $C_j = \widetilde{C_j} / \ZZ^2$ is a quasi-isometric to a compact torus, hence  $X = \widetilde X/\Gamma$ is compact as well.
\end{proof}

\begin{proof}[Proof of \refcor{Cubulation}]
%
Let $V_1, \ldots, V_n$ be the cusps of $M$. For every $V_j$, choose two distinct slopes $\alpha(V_j)$ and $\beta(V_j)$. For each $V_j$, let $S_{\alpha,j}$ and $S_{\beta,j}$ be two quasi-Fuchsian surfaces produced by \refthm{QFslope}, whose immersed boundary slopes are $\alpha(V_j)$ and $\beta(V_j)$ respectively. Let $\Scal$ be the (infinite) collection of QF surfaces consisting of $\{ S_{\alpha,1}, \ldots, S_{\alpha,n}, \, S_{\beta,1}, \ldots, S_{\beta,n} \}$ and the ubiquitous collection of closed QF surfaces guaranteed by \refthm{ClosedQFSurfaces}. 
    In particular, every $S \in \Scal$ is either closed or has all cusps mapping to exactly one multiple of  $\alpha(V_j)$ or $\beta(V_j)$, for one $V_j$.

Now, \refthm{ProperCocompact} applied to $\Scal$ gives a free and cocompact action on a dual cube complex $\widetilde X$.
\end{proof}

 \small
 
\bibliography{refs.bib} 

\bibliographystyle{amsplain}

\end{document}